\documentclass[12pt]{article}
\usepackage{float}
\usepackage{makeidx}
\usepackage{latexsym}

\usepackage[margin=2.5cm]{geometry}
\usepackage[dvips]{graphicx}

\usepackage[english]{babel}
\usepackage{amssymb}
\usepackage{amsmath}
\usepackage{amsthm}
\usepackage{lineno}
\usepackage{xcolor}
\definecolor{blau}{rgb}{0.1,0.0,0.9}
\definecolor{gruen}{cmyk}{1.0,0.2,0.7,0.07}
\definecolor{mag}{cmyk}{0.0,0.9,0.3,0.0}

\newtheorem{theorem}{Theorem}[section]
\newtheorem{lemma}[theorem]{Lemma}

\newtheorem{conjecture}[theorem]{Conjecture}

\newtheorem{property}[theorem]{Property}
\theoremstyle{definition}

\newtheorem{definition}[theorem]{Definition}

\begin{document}
\date{\today}
\title{Latin cubes of even order with forbidden entries}

\author{
{\sl Carl Johan Casselgren}\footnote{Department of Mathematics, Link\"oping University, SE-581 83 Link\"oping, Sweden. \newline
{\it E-mail address:} carl.johan.casselgren@liu.se  \, Casselgren was supported by a grant from the Swedish
Research Council (2017-05077)}
\and {\sl Lan Anh Pham}\footnote{Department of Mathematics, 
Ume\aa\enskip University, 
SE-901 87 Ume\aa, Sweden.
{\it E-mail address:} lan.pham@umu.se
}
}
\maketitle

\bigskip
\noindent
{\bf Abstract.}
We consider the problem of constructing Latin cubes subject to the condition
that some symbols may not appear in certain cells. 
We prove that there is
a constant $\gamma > 0$ such that if $n=2t$ and $A$ is a
$3$-dimensional $n\times n\times n$ array where every cell contains at most
$\gamma n$ symbols, and every symbol occurs at most $\gamma n$ times in 
every line of $A$, then $A$ is {\em avoidable}; that is, there is
a Latin cube $L$ of order $n$ such that for every $1\leq i,j,k\leq n$,
the symbol in position $(i,j,k)$ of $L$ does not appear in the corresponding
cell of $A$.

\bigskip


\section{Introduction}
Consider an $n \times n$ array $A$ in which
every cell $(i,j)$
contains a subset $A(i,j)$ of the symbols in $[n]=\{1, \dots,n \}$.
If every cell contains at most $m$ symbols, and every symbol occurs
at most $m$ times in every row and column, then $A$ is an {\em $(m,m,m)$-array}.
Confirming a conjecture by H\"aggkvist \cite{Haggkvist}, it was proved in
\cite{AndrenCasselgrenOhman} that there is a constant $c>0$ such that if
$m \leq cn$ and $A$ is an $(m,m,m)$-array, then $A$
 is {\em avoidable};
that is, there is a Latin square $L$ such that for every $(i,j)$ the symbol in
position $(i,j)$ in $L$ is not in $A(i,j)$
(see also \cite{Lina, AndrenCasselgrenMarkstrom}).
In this paper we consider the corresponding problem for Latin cubes.

In general, a {\em cube} is just a
$3$-dimensional array having
layers stacked on top of each other.
Now, a cube has {\em lines} in three directions
obtained from fixing two coordinates
and allowing the third to vary. The lines obtained by varying the 
first, second, and third coordinates will be 
referred to respectively as {\em columns}, {\em rows}, and {\em files}. 
The first,
second, and third coordinates themselves will be referred to
as the indices of the rows, columns, and files.

A \textit{Latin cube} $L$ of order $n$ on the symbols $\{1,\dots,n\}$ 
is an $n \times n \times n$ cube  such that each symbol in $\{1,\dots,n\}$ appears
exactly once in each row, column and file. The symbol in position $(i,j,k)$
of $L$ is denoted by $L(i,j,k)$.    Latin cubes have been studied by a number of authors, both with respect to enumeration
and e.g. extension from partial cubes.  An extensive survey of early results can be found in \cite{MckWan}.

An  $n \times n \times n$ cube where 
each cell contains a subset of the symbols in the set $\{1,\dots,n\}$ is called an 
\textit{$(m,m,m,m)$-cube (of order $n$)} if the following conditions hold:
\begin{itemize}
\item[(a)] No cell contains a set with more than $m$ symbols. 
\item[(b)] Each symbol occurs at most $m$ times in each row.
\item[(c)] Each symbol occurs at most $m$ times in each column.
\item[(d)] Each symbol occurs at most $m$ times in each file.
\end{itemize}

Let $A(i,j,k)$ denote the set of symbols in the cell $(i,j,k)$ of $A$. 
If we simplify notation, and write $A(i,j,k)=q$ if the set of 
symbols in cell $(i,j,k)$ of $A$ is $\{q\}$, then 
a $(1,1,1,1)$-cube is a {\em partial Latin cube},
and a \textit{Latin cube} $L$ is simply a $(1,1,1,1)$-cube
with no empty cell.

Given an $(m,m,m,m)$-cube $A$ of order $n$, a Latin cube $L$
of order $n$ \textit{avoids} $A$ if
there is no cell $(i,j,k)$ 
of $L$ such that $L(i,j,k) \in A(i,j,k)$;
if there is such a Latin cube, then
$A$ is \textit{avoidable}. 

Problems on extending partial Latin cubes have been studied for a long time, with the earliest results appearing in the 1970s \cite{Cruse}; in the more recent literature we have  \cite{Bri, Bryant,KuhlDenley,DenleyOhman}. The study of the more general problem of constructing Latin cubes subject to the condition that some
symbols cannot appear in certain cells was
quite recently intiated in \cite{JohanKlasLan}:
building on earlier research \cite{AndrenCasselgrenOhman, Lina, Andren2010latin},
the authors of the present paper along with Markstr\"om obtained an analogue
of the result of \cite{Lina}, which considered Latin squares, for Latin cubes:

\begin{theorem}
\label{th:prev}
\cite{JohanKlasLan}
	There is a constant $\gamma > 0$ such that for every $n=2^t$ $(t \geq 30)$ any
	$(\gamma n, \gamma n, \gamma n, \gamma n)$-cube of order $n$
is avoidable.
\end{theorem}

In general we conjecture that the following is true:

\begin{conjecture}
\label{conj:main}
	There is a constant $\gamma >0$ such that for every positive integer $n$ 
	any $(\gamma n, \gamma n, \gamma n, \gamma n)$-cube of order $n$ is avoidable.
\end{conjecture}

If true, then $\gamma \leq 1/3$, as follows by an example from \cite{JohanKlasLan}.

In this paper we consider the problem of avoiding $(m,m,m,m)$-cubes of even order;
generalizing Theorem \ref{th:prev} we confirm Conjecture
\ref{conj:main} for the case $n=2t$.
Our main result is the following, which establishes an analogue of
a main result of \cite{Andren2010latin}, 
which considered Latin squares, for Latin cubes.

\begin{theorem}
\label{maintheorem}
There is a positive constant $\gamma$ such that if $t \geq 2^{30}$ and 
$m \leq \gamma 2t$, then any $(m,m,m,m)$-cube $A$ of order $2t$ is avoidable.
\end{theorem}

To prove this result we need to significantly
extend the machinery from \cite{JohanKlasLan}. 
In Section 2 we give some definitions and preparatory lemmas, generalizing
several tools from \cite{JohanKlasLan}, 
and in Section 3 we prove Theorem \ref{maintheorem}.

Let us finally remark that the main result of this paper, as well as the problem of 
extending 
partial Latin cubes, can be recast as list edge coloring problems on the complete 3-uniform
3-partite hypergraph $K^3_{n,n,n}$.  Problems on extending partial edge colorings for 
ordinary graphs have been studied to some extent, cf. \cite{hypercube, EGHKPS},
but similar problems for hypergraphs remain mostly unexplored.


\section{Definitions and properties of starting Latin cubes}

In this section we give some definitions and collect essential
properties of starting Latin cubes.

Let $A$ be an $n \times n \times n$ cube.
Given $i \in [n]$, \textit{row layer} $i$ in 
$A$ is a set of $n^2$ cells $\{(i,j^*,k^*) : j^* \in [n], k^* \in [n] \}$;
given $j \in [n]$, \textit{column layer} $j$ in 
$A$ is a set of $n^2$ cells $\{(i^*,j,k^*) : i^* \in [n],k^* \in [n]\}$;
given $k \in [n]$, \textit{file layer} $k$  in $A$
is a set of $n^2$ cells $\{(i^*,j^*,k) : i^* \in [n], j^* \in [n]\}$. 
As mentioned above, by fixing two coordinates and varying the third,
we obtain rows, columns and files of a $n \times n \times n$ cube.
Formally we define a row of such a cube $A$ as a set of cells
$R_{i,k} = \{(i,j^*,k) : j^* \in [n]\}$, a column as the set
$C_{j,k} = \{(i^*,j,k) : i^* \in [n]\}$, and files
$F_{i,j} = \{(i,j,k^*) : k^* \in [n]\}$.

In this paper, we shall assume $x \mod t =t$ in the case when 
$x \mod t \equiv 0$.
%
The following Latin square seems to have been first employed by H\"aggkvist
and is a useful tool for completing sparse partial Latin squares.
(see e.g. \cite{AndrenCasselgrenMarkstrom}).

\begin{definition} 
\label{defSLsquare}
 For $1 \leq i, j \leq n$, the {\em starting Latin square} of order 
$n=2t$ is defined as follows.

\begin{equation}
L(i,j) = \left\{ \begin{array}{rcl}
j-i +1 \mod t & \mbox{for}  & i,j \leq t,\\
i-j +1 \mod t & \mbox{for} & i, j> t, \\
(j-i+1 \mod t) + t & \mbox{for} & i \leq t, j>t,\\
(i-j+1 \mod t) + t & \mbox{for} & i > t, j \leq t.
\end{array}\right.
\end{equation}
\end{definition}

Define 
$\mathcal{S}_1=\{1,2,\dots,t\}$ and $\mathcal{S}_2=\{t+1,\dots,n\}$.
We can divide the starting Latin square into four quadrants such that each quadrant corresponds to one of the four cases in 
Definition \ref{defSLsquare}; so the set of symbols in each quadrant 
is $S_1$ or $S_2$;
for example, the set of symbols in the quadrant created by $t^2$ cells 
$(i,j)$ satisfying $i,j \leq t$ is $S_1$.

A {\em $4$-cycle} (or {\em intercalate})
in a Latin square $L$ is a set of four cells 
$\{(i_1,j_1), (i_1,j_2), (i_2,j_1), (i_2,j_2)\}$ such that 
$L(i_1,j_1)=L(i_2,j_2)$ and $L(i_1,j_2)=L(i_2,j_1)$.
We note some important properties of starting Latin squares (cf. \cite{Andren2010latin}, \cite{Padraic}).

\begin{property}
\label{prop:cycles}
Each cell in the $2t \times 2t$
starting Latin square is in $t$ $4$-cycles;
each of these $4$-cycles contains two symbols $s_1 \in \mathcal{S}_1$ and $s_2 \in \mathcal{S}_2$.
Permuting the rows, the columns
or the symbols does not affect the number of $4$-cycles that a cell is part of.
\end{property}

If two quadrants of the starting Latin square contain the same set of symbols, then
they are {\em opposite}, otherwise they are {\em adjacent}.

\begin{property}
\label{pro4cyc}
Each quadrant in the starting Latin square contains exactly one cell
from each of the $t$ $4$-cycles from the previous property.
Moreover, two cells in the same row $($or the same column$)$, but from 
adjacent quadrants, uniquely determine
a $4$-cycle. 
\end{property}

\begin{property}
\label{two4cycle}
The intersection of two $4$-cycles is either empty, or it contains $1$ or $4$ cells.
\end{property}
Given an integer $n=2t$, and $1 \leq i, j, k \leq n$, 
we define the \textit{starting Latin cube} similarly to the starting Latin square. 
\begin{definition}
\label{defSLatincube}
The {\em starting Latin cube} $L$ of order $n=2t$ on the symbols $\{1,\dots,n\}$ 
is an $n\times n \times n$ Latin cube such that
\begin{equation}
L(i,j,k) = \left\{ \begin{array}{rcl}
j-i +k \mod t & \mbox{for}  & i,j,k \leq t \quad (\textrm{and for} \quad  i,k > t, j\leq t),\\
i-j +k \mod t & \mbox{for} & i, j>t, k\leq t \quad (\textrm{and for} \quad  k, j> t, i\leq t),\\
(j-i+k \mod t) + t & \mbox{for} & i,k \leq t, j>t \quad (\textrm{and for} \quad  i, j,k > t),\\
(i-j+k \mod t) + t & \mbox{for} & k, j \leq t, i > t \quad (\textrm{and for} \quad  i, j \leq t, k> t).\\
\end{array}\right.
\end{equation}
\end{definition}

We can divide the starting Latin cube into eight octants such that each octant corresponds to one of the eight cases in 
Definition \ref{defSLatincube}; so the set of symbols in each octant is $S_1$ or $S_2$.
For example, the set of symbols in the octant created by the $t^3$ cells 
$(i,j,k)$
satisfying $i,j,k \leq t$ is $S_1$.
Given any octant $O_1$, it is easy to check that there is only one octant $O_2$ such that 
for every $(i_1,j_1,k_1) \in O_1$ and every $(i_2,j_2,k_2) \in O_2$ we have $(i_1 - i_2)(j_1 - j_2)(k_1 - k_2) \neq 0$;
$O_1$ and $O_2$ are called {\em opposite} octants. Note that sets of symbols
used in opposite octants are different.

In the following, if two cells are in the same row, column or file layer,
then we just say that they are in the same layer.
The intersection between an octant and a layer in the starting Latin cube
corresponds to a quadrant of the starting Latin square (by possibly relabeling
symbols).
As for the starting Latin square, there are four quadrants in each layer 
and two quadrants in a layer are opposite if they 
contain the same set of symbols; otherwise, they are adjacent.

For future reference, we state the following two important properties
of the starting Latin cube.

\begin{property}
\label{cellproperty}
Given two cells $(i_1,j_1,k_1)$ and $(i_1,j_2,k_2)$ in two opposite quadrants
of the starting Latin cube $L$, 
if $L(i_1,j_1,k_1) \equiv i_1-j_1 +k_1 \mod t$, then $L(i_1,j_2,k_2) \equiv j_2-i_1 +k_2 \mod t$;
if  $L(i_1,j_1,k_1) \equiv j_1-i_1 +k_1 \mod t$, then $L(i_1,j_2,k_2) \equiv i_1-j_2 +k_2 \mod t$.
A similar property holds for two cells $(i_1,j_1,k_1)$ and $(i_2,j_2,k_1)$ 
lying in
two opposite quadrants. 

Moreover, for two cells $(i_1,j_1,k_1)$ and $(i_2,j_1,k_2)$ from 
two opposite quadrants, 
if $L(i_1,j_1,k_1) \equiv i_1-j_1 +k_1 \mod t$, then $L(i_2,j_1,k_2) \equiv i_2-j_1 +k_2 \mod t$;
if  $L(i_1,j_1,k_1) \equiv j_1-i_1 +k_1 \mod t$, then $L(i_2,j_1,k_2) \equiv j_1-i_2 +k_2 \mod t$. 
\end{property}

\begin{property}
\label{procubes}
Properties \ref{prop:cycles}, \ref{pro4cyc} and \ref{two4cycle}
hold for all row, column and file layers of the starting Latin cube.
\end{property}

\begin{proof}
By Definitions \ref{defSLsquare} and \ref{defSLatincube}, it is straightforward that every file layer of the starting Latin cube has 
the properties stated in Properties
 \ref{prop:cycles}, \ref{pro4cyc} and \ref{two4cycle}. Consider a row layer $i_1$ and an arbitrary cell $(i_1, j_1,k_1)$. 
Let us consider the case $L(i_1,j_1,k_1) \equiv i_1 -j_1 +k_1 \mod t$
(the case $L(i_1,j_1,k_1) \equiv j_1 -i_1 +k_1 \mod t$ can be done similarly).
In the row layer $i_1$, let $L_1$ be the quadrant that contains $(i_1, j_1,k_1)$ and $L_2$ 
 be the quadrant that is adjacent to $L_1$ and has the same rows as $L_1$;
choose any cell $(i_1,j_2,k_1) \in R_{i_1,k_1} \cap L_2$.
By Definition \ref{defSLatincube} we have that
$L(i_1, j_2,k_1) \equiv i_1 -j_2+k_1 \mod t$.
Choose the cell $(i_1,j_1,k_2)$ such that $L(i_1,j_1,k_2)=L(i_1,j_2,k_1)$;
then $(i_1,j_1,k_2)$ must be in the quadrant
$L_3$ that is opposite to $L_2$, and this implies that 
the cell $(i_1,j_2,k_2)$ must be in the quadrant $L_4$ that is opposite to $L_1$. Using Property
\ref{cellproperty}, we deduce that $L(i_1,j_2,k_2) \equiv j_2 - i_1+k_2 \mod t$ and $L(i_1,j_1,k_2) \equiv j_1 - i_1 + k_2 \mod t$.
Hence, $$L(i_1,j_2,k_2) - L(i_1,j_1,k_1) \equiv (j_2 - i_1+k_2) - (i_1 -j_1 +k_1)$$
$$\equiv (j_1 - i_1 + k_2) - (i_1 - j_2 + k_1) \equiv 
L(i_1,j_1,k_2) - L(i_1,j_2,k_1) \equiv 0 \mod t.$$

Since $(i_1,j_2,k_2)$ and $(i_1, j_1,k_1)$ are in two opposite quadrants, they contain the same symbol,
so $L(i_1,j_2,k_2) = L(i_1,j_1,k_1)$ or, in other words,
$\{(i_1, j_1,k_1), (i_1,j_2,k_1), (i_1,j_1,k_2), (i_1,j_2,k_2\}$ is
a $4$-cycle and, more generally, two cells from the same row 
in the row layer $i_1$, but from two adjacent quadrants, 
uniquely determine a $4$-cycle. 
Similarly, we can also prove that two cells in the same file of 
row layer $i_1$, but from two adjacent quadrants,
uniquely determine a $4$-cycle.

If the intersection of two $4$-cycles contains $2$ (or $3$ cells), then
two cells of one of the $4$-cycles have distinct column
and file coordinates.
Assume that these cells are $(i_1,j_1,k_1)$ and $(i_1,j_2,k_2)$ ($j_1 \neq j_2$, $k_1 \neq k_2$). But then the intersection of the two
$4$-cycles must be the $4$ cells $(i_1,j_1,k_1), $ $ (i_1,j_2,k_1), $ $ (i_1,j_1,k_2), $ $ (i_1,j_2,k_2)$, a contradiction.
We conclude that the intersection of two $4$-cycles in row layer $i_1$ is either empty, or it contains $1$ or $4$ cells.
Furthermore, since the sets of symbols in two adjacent quadrants are different and $|R_{i_1,k_1} \cap L_2|=t$,
any cell $(i_1, j_1,k_1)$ in the row layer $i_1$ in the 
starting Latin cube is in $t$ $4$-cycles, and
each of these $4$-cycles contains two symbols $s_1 \in \mathcal{S}_1$ and $s_2 \in \mathcal{S}_2$.

Permuting the rows, the files or the symbols does not affect the number of $4$-cycles in row layer $i_1$ that a cell is part of.
It follows that every row layer of the starting Latin cube has the 
properties stated in Properties
 \ref{prop:cycles}, \ref{pro4cyc} and \ref{two4cycle}. 
That the same holds for every
column layer of the starting Latin cube can be proved similarly.

\end{proof}

\begin{definition}
\label{defsubcube}
A {\em subcube of order $2$} (or just {\em subcube}) 
in a Latin cube $L$ is a set of eight cells 
$$\{(i_1,j_1,k_1), (i_1,j_2,k_1), (i_2,j_1,k_1), (i_2,j_2,k_1), 
(i_1,j_1,k_2), (i_1,j_2,k_2), (i_2,j_1,k_2), (i_2,j_2,k_2)\}$$
such that $$L(i_1,j_1,k_1)=L(i_2,j_2,k_1)=L(i_1,j_2,k_2)=L(i_2,j_1,k_2)$$
and $$L(i_1,j_2,k_1)= L(i_2,j_1,k_1)=L(i_1,j_1,k_2)= L(i_2,j_2,k_2).$$
\end{definition}

\begin{property}
\label{tsubcube}
Each cell in the starting Latin cube of order $n=2t$ belongs to $t$ subcubes;
each of these subcubes contains two distinct 
symbols $s_1 \in \mathcal{S}_1$ and $s_2 \in \mathcal{S}_2$.
\end{property}
\begin{proof}
Consider an arbitrary cell $(i_1,j_1,k_1)$ of the starting Latin cube  $L$ which belongs to 
a $4$-cycle $\mathfrak{c}_1=\{(i_1,j_1,k_1), $ $ (i_1,j_2,k_1), $ $ (i_2,j_1,k_1), $ $ (i_2,j_2,k_1)\}$ such that
$L(i_1,j_1,k_1)=L(i_2,j_2,k_1)$ and $L(i_1,j_2,k_1)= L(i_2,j_1,k_1)$. By Property \ref{procubes}, 
there are $t$ $4$-cycles $\mathfrak{c}_1$ in the file layer $k_1$ containing $(i_1,j_1,k_1)$, 
each of these $4$-cycles contains two symbols $s_1 \in \mathcal{S}_1$ and $s_2 \in \mathcal{S}_2$.

Let us consider the case $L(i_1,j_1,k_1) \equiv i_1 -j_1 +k_1 \mod t$
(the case $L(i_1,j_1,k_1) \equiv j_1 -i_1 +k_1 \mod t$ can be done similarly).
Since the two cells $(i_1,j_1,k_1)$ and $(i_1,j_2,k_1)$ are in the same 
row $R_{i_1,k_1}$, 
by Definition \ref{defSLatincube}, 
$L(i_1,j_2,k_1) \equiv i_1 -j_2 +k_1 \mod t$.
By Property \ref{cellproperty} and Property \ref{procubes}, we have 
that $L(i_2,j_2,k_1) \equiv j_2 -i_2+k_1 \mod t$,
$L(i_2,j_1,k_1) \equiv j_1 -i_2+k_1 \mod t$. 
Moreover, the two cells $(i_1,j_1,k_1)$ and $(i_2,j_1,k_1)$ define a unique $4$-cycle
$$\mathfrak{c}_2=\{(i_1,j_1,k_1), (i_2,j_1,k_1), (i_1,j_1,k_2), (i_2,j_1,k_2)\}$$ 
in the column layer $j_1$
such that $L(i_1,j_1,k_1)=L(i_2,j_1,k_2)$ and $L(i_2,j_1,k_1)= L(i_1,j_1,k_2)$. 
Since $(i_2,j_1,k_2)$ and $(i_1,j_1,k_1)$ are in two opposite quadrants in the column layer $j_1$,
Property \ref{cellproperty} implies that $L(i_2,j_1,k_2) \equiv i_2 -j_1+k_2 \mod t$.
It is easy to check that $L(i_1,j_2,k_2) \equiv j_2-i_1+k_2 \mod t$,
$L(i_1,j_1,k_2) \equiv j_1 -i_1+k_2 \mod t$, $L(i_2,j_2,k_2) \equiv i_2-j_2+k_2 \mod t$.
Hence
$$i_1 - j_1 +k_1  \equiv j_2 -i_2 +k_1 \equiv i_2 - j_1 +k_2 \mod t$$ 
and 
$$i_1 - j_2 + k_1 \equiv j_1 - i_2 +k_1 \equiv j_1 - i_1 + k_2 \mod t$$
Thus, we have
$$j_1 - i_1 + k_2  \equiv i_2 -j_2 +k_2 \equiv  j_1 - i_2 +k_1 \mod t$$
and 
$$j_2 -i_1+ k_2  \equiv i_2 - j_1 +k_2 \equiv  i_1 - j_1 +k_1  \mod t.$$

Since the set of symbols in two opposite quadrants
are the same, we obtain the following
$$L(i_1,j_2,k_1)= L(i_2,j_1,k_1)=L(i_1,j_1,k_2)= L(i_2,j_2,k_2)$$
and $$L(i_1,j_1,k_1)=L(i_2,j_2,k_1)=L(i_1,j_2,k_2)=L(i_2,j_1,k_2).$$ 
This implies that
$$\{(i_1,j_1,k_1), (i_1,j_2,k_1), (i_2,j_1,k_1),  (i_2,j_2,k_1), (i_1,j_1,k_2), 
(i_1,j_2,k_2), (i_2,j_1,k_2),  (i_2,j_2,k_2)\}$$ is a subcube;
and so, each cell in the starting Latin cube belongs to 
$t$ subcubes; each of these subcubes contains 
two distinct symbols $s_1 \in \mathcal{S}_1$ and $s_2 \in \mathcal{S}_2$.
\end{proof}
Note that some cells in the starting Latin cube may be contained in a 
subcube with symbols only in one of the sets $\mathcal{S}_1$ (or  $\mathcal{S}_2$).
Since the number of such subcubes is very small, from now on,
when referring to a subcube we mean a subcube containing two symbols
$s_1 \in \mathcal{S}_1$ and $s_2 \in \mathcal{S}_2$.
The following property is straightforward
from the proof of Property \ref{tsubcube}.

\begin{property}
\label{subcubecells}
Let  $(i_1,j_1,k_1)$ and $(i_2,j_2,k_2)$ 
$(i_1 \neq i_2, j_1 \neq j_2, k_1 \neq k_2)$ 
be two cells in a subcube in the starting Latin cube of order $n$, 
then the octant containing $(i_1,j_1,k_1)$ is opposite to the octant containing $(i_2,j_2,k_2)$.
Furthermore, if $L(i_1,j_1,k_1) \equiv i_1 - j_1 + k_1 \mod t$, then $L(i_2,j_2,k_2) \equiv i_2 - j_2 + k_2 \mod t$;
if $L(i_1,j_1,k_1) \equiv j_1 - i_1 + k_1 \mod t$, then $L(i_2,j_2,k_2) \equiv j_2 - i_2 + k_2 \mod t$.
\end{property}

\begin{property}
\label{cellsinoppositequadrants}
Let $Q_1$ and $Q_2$ be two opposite quandrants in the row layer $i$
in the starting Latin cube $L$ of order $n$. Moreover, let $(i,j_1,k_1) \in Q_1$ and
$(i,j_2,k_2) \in Q_2$. If $L(i,j_1,k_1)=L(i,j_2,k_2)$,
then there is a subcube containing $(i,j_1,k_1)$ and $(i,j_2,k_2)$.
A similar property holds for column layers and file layers.
\end{property}
\begin{proof}
This property is obvious from the fact that there are exactly $t$ subcubes
containing $(i,j_1,k_1)$;
each of these subcubes must contain a cell $(i,j,k)$ in $Q_2$ such that
$L(i,j,k)=L(i,j_1,k_1)$, and in $Q_2$ 
there are exactly $t$ cells with the symbol $L(i,j_1,k_1)$.
\end{proof}

\begin{property}
The intersection of two subcubes in a Latin cube
is either empty, or it contains $1$ or $8$ cells.
\end{property}
\begin{proof}
Assume that the intersection of two given subcubes contains at least $2$ cells. 
If these $2$ cells lie in a $4$-cycle of a layer of the Latin cube,
then by  Property \ref{two4cycle}, 
this $4$-cycle belongs to the  intersection of two subcubes.
But each $4$-cycle defines a unique subcube, which implies that 
the intersection of the two subcubes
contains $8$ cells. If these $2$ cells do not lie in a $4$-cycle, then
they must have distinct row, column
and file coordinates, so if we denote these two cells by
$(i_1,j_1,k_1)$ and $(i_2,j_2,k_2)$, respectively, then
$i_1 \neq i_2$, $j_1 \neq j_2$, $k_1 \neq k_2$. Hence, the intersection of 
the two subcubes must be the
$8$ cells $(i_1,j_1,k_1), $ $ (i_1,j_2,k_1), $ $ (i_2,j_1,k_1), $ $ (i_2,j_2,k_1),$ $ (i_1,j_1,k_2), $ 
$(i_1,j_2,k_2), $ $ (i_2,j_1,k_2), $ $  (i_2,j_2,k_2)$.
\end{proof}

\begin{definition}
Given a subcube $$\mathcal{C}=\{(i_1,j_1,k_1), (i_1,j_2,k_1), 
(i_2,j_1,k_1), (i_2,j_2,k_1), (i_1,j_1,k_2), (i_1,j_2,k_2), 
(i_2,j_1,k_2),  (i_2,j_2,k_2)\}$$ in a Latin cube $L$, 
a \textit{swap on $\mathcal{C}$} (or simply a \textit{swap})
denotes the transformation $L \rightarrow L'$ which retains the content
of all cells of $L$ except that if 
$$L(i_1,j_1,k_1)=L(i_2,j_2,k_1)=L(i_1,j_2,k_2)=L(i_2,j_1,k_2)=x_1$$
and $$L(i_1,j_2,k_1)= L(i_2,j_1,k_1)=L(i_1,j_1,k_2)= L(i_2,j_2,k_2)=x_2$$ then
$$L'(i_1,j_1,k_1)=L'(i_2,j_2,k_1)=L'(i_1,j_2,k_2)=L'(i_2,j_1,k_2)=x_2$$
and $$L'(i_1,j_2,k_1)= L'(i_2,j_1,k_1)=L'(i_1,j_1,k_2)= L'(i_2,j_2,k_2)=x_1.$$
\end{definition}

\begin{definition}
\label{halftransversalset}
A \textit{half transversal-set} $T_h$ determined by a cell $(i_1,j_1,k_1)$ in the starting 
Latin cube of order $n$ is a set of $t$ cells $(i_2,j_2,k_2)$ such that
$$\mathcal{C}=\{(i_1,j_1,k_1),  (i_1,j_2,k_1), (i_2,j_1,k_1), 
(i_2,j_2,k_1), (i_1,j_1,k_2), (i_1,j_2,k_2), (i_2,j_1,k_2), (i_2,j_2,k_2)\}$$
is a subcube. 
\end{definition}

It is straightforward that $(i_1,j_1,k_1)$ is in the half transversal-set determined by
$(i_2,j_2,k_2)$.

\begin{property}
\label{symbolsetofhalftransversalset}
The set of symbols appearing in the cells in a half transversal-set in the starting 
Latin cube of order $n=2t$ is $\mathcal{S}_1=\{1,2,\dots,t\}$ 
or $\mathcal{S}_2=\{t+1, \dots,n\}$.
\end{property}
\begin{proof}
Let $T_h$ be the half transversal-set determined by a cell $(i_1,j_1,k_1)$ in the starting
Latin cube of order $n$. Consider an arbitrary cell $(i_2,j_2,k_2) \in T_h$.
By Definition \ref{halftransversalset}, we have 
that $L(i_2,j_2,k_2)=L(i_1,j_2,k_1)$.
Since the row $R_{i_1,k_1}$ contains $n$ different symbols, there are 
no two cells  in $T_h$ that contain the
same symbol. Moreover, by Property \ref{tsubcube}, 
if $L(i_1,j_1,k_1) \in \mathcal{S}_1$ then 
$L(i_2,j_2,k_2)=L(i_1,j_2,k_1) \in \mathcal{S}_2$, which implies 
that the set of symbols appearing in the cells in
the half transversal-set $T_h$ is $\mathcal{S}_2$. Similarly, 
if $L(i_1,j_1,k_1) \in \mathcal{S}_2$ then 
the set of symbols appearing in the cells in the half transversal-set $T_h$ 
is $\mathcal{S}_1$.
\end{proof}

\begin{property}
\label{twocellsinhalftransversalset}
Let $(i,j,k)$ be any cell in the half transversal-set 
determined by $(i_1,j_1,k_1)$ in the starting Latin cube of order $n$.
There is a subcube containing $(i,j,k)$ and
$(i_2,j_2,k_2)$, $i \neq i_2, j \neq j_2, k\neq k_2$,
\emph{(}which also means that $(i_1,j_1,k_1)$ and $(i_2,j_2,k_2)$ are in the half transversal-set 
determined by $(i,j,k)$\emph{)}  if and only if $(i_1,j_1,k_1)$ and $(i_2,j_2,k_2)$ are in the same octant and 
\begin{itemize}
\item $i_1-i_2 \equiv j_1-j_2 \equiv k_2 - k_1 \mod t$  \emph{if} $L(i_1,j_1,k_1) \equiv i_1 - j_1 + k_1 \mod t$; \emph{or}
\item  $i_1-i_2 \equiv j_1-j_2 \equiv k_1 - k_2 \mod t$ \emph{if} $L(i_1,j_1,k_1) \equiv j_1 - i_1 + k_1 \mod t$.
\end{itemize}
\end{property}
\begin{proof}
$\Rightarrow$ Assume that  $(i,j,k)$ and $(i_2,j_2,k_2)$ $(i \neq i_2, j \neq j_2, k\neq k_2)$ are two cells of a subcube.
Property \ref{subcubecells} implies that both $(i_1,j_1,k_1)$ and $(i_2,j_2,k_2)$ belong to the octant that is opposite
to the octant containing $(i,j,k)$. Moreover, Definition \ref{defsubcube} 
implies that
$$L(i,j,k)=L(i,j_1,k_1)=L(i,j_2,k_2)=L(i_1,j,k_1)=L(i_2,j,k_2)=L(i_1,j_1,k)=L(i_2,j_2,k).$$
We now consider two different cases.

\begin{itemize}
\item If $L(i_1,j_1,k_1) \equiv i_1 - j_1 + k_1 \mod t$, then $L(i,j,k) \equiv i - j + k \mod t$ by Property \ref{subcubecells}.
Now, according to Property \ref{cellproperty}, we have that
$$i-j+k \equiv j_1 - i + k_1 \equiv j_2 - i + k_2 \equiv i_1 - j + k_1 \equiv 
i_2 - j + k_2 \equiv j_1 - i_1 + k \equiv j_2 - i_2 +k \mod t,$$ that is,
$i_1-i_2 \equiv j_1-j_2 \equiv k_2 - k_1 \mod t$.

\item If $L(i_1,j_1,k_1) \equiv j_1 - i_1 + k_1 \mod t$, then $L(i,j,k) \equiv j - i + k \mod t$ by Property \ref{subcubecells}.
Moreover, by Property \ref{cellproperty}, we have that
$$j-i+k \equiv i - j_1 + k_1 \equiv i - j_2 + k_2 \equiv j - i_1 + k_1 \equiv 
j- i_2+ k_2 \equiv i_1 - j_1 + k \equiv i_2 - j_2 +k \mod t,$$ that is,
$i_1-i_2 \equiv j_1-j_2 \equiv k_1 - k_2 \mod t$.

\end{itemize}

$\Leftarrow$ Now assume that $(i_1,j_1,k_1)$ and $(i_2,j_2,k_2)$ are in the same octant, $L(i_1,j_1,k_1) \equiv i_1 - j_1 + k_1 \mod t$,
$i_1-i_2 \equiv j_1-j_2 \equiv k_2 - k_1 \mod t$. This implies 
that $L(i_2,j_2,k_2) \equiv i_2 - j_2 +k_2 \mod t$.
We also have that $i \neq i_2, j \neq j_2, k\neq k_2$,
since $(i_2,j_2,k_2)$ and $(i,j,k)$ are in two opposite octants.
Furthermore, since $(i,j,k)$ and $(i_1,j_1,k_1)$ are two cells of a subcube, 
it follows from 
Definition \ref{defsubcube}, Property \ref{cellproperty} and Property \ref{procubes}
that $L(i_1,j_1,k_1)=L(i_1,j,k)$ or $i_1 - j_1 + k_1 \equiv j - i_1 + k 
\mod t$. 
Thus $$j + k - i_2 \equiv i_1 - j_1 + k_1 + i _1 - i_2 \equiv (i_1 - j_1) + (k_1 + i _1) - i_2 
\equiv (i_2 - j_2) + (k_2 + i_2) - i_2 \equiv i_2 - j_2 + k_2 \mod t.$$
Since  $(i',j_1,k_1)$ and $(i',j_2,k_2)$ are in the same quadrant for every row layer $i'$
and $(i',j_1,k_1)$ and $(i',j,k)$ are in two opposite quadrants for every row layer $i'$,
$(i_2,j_2,k_2)$ and $(i_2,j,k)$ are in two opposite quadrants in the row layer $i_2$.
Thus by Property \ref{cellproperty}, 
$$L(i_2,j,k) \equiv j - i_2 + k \equiv L(i_2,j_2,k_2) \mod t.$$ Hence 
$L(i_2,j,k) =L(i_2,j_2,k_2)$ since two opposite quadrants contain the same set of symbols. 
Similarly, one can prove that 
$$L(i,j_2,k)=L(i,j,k_2)=L(i_2,j_2,k_2), \text{ and }$$ 
$$L(i,j_2,k_2)=L(i_2,j,k_2)=L(i_2,j_2,k)=L(i,j,k),$$ which implies that $(i_2,j_2,k_2)$ and $(i,j,k)$ are two cells of a subcube, and so,
$(i_2,j_2,k_2)$ is in the half transversal-set determined by $(i,j,k)$.

The proof in the case when $(i_1,j_1,k_1)$ and $(i_2,j_2,k_2)$ are in the same octant,
$L(i_1,j_1,k_1) \equiv j_1 - i_1 + k_1 \mod t$ and $i_1-i_2 \equiv j_1-j_2 \equiv k_1 - k_2 \mod t$
is similar.
\end{proof}

\begin{property}
\label{twohalftransversalset}
Two different
half transversal-sets in the starting Latin cube of order $n$ are disjoint.
\end{property}
\begin{proof}
Let $T_{h_1}$ be the half transversal-set determined by $(i_1,j_1,k_1)$
and $T_{h_2}$ be the half transversal-set determined by $(i_2,j_2,k_2)$.
Assume further that the two distinct cells $(i_1,j_1,k_1)$ and $(i_2,j_2,k_2)$ 
do not determine the same half transversal-set
(i.e. $T_{h_1} \neq T_{h_2}$) and  $T_{h_1} \cap T_{h_2} \neq \emptyset$.
Pick a cell $(i,j,k) \in T_{h_1} \cap T_{h_2}$, and let us consider the case 
$L(i,j,k) \equiv i -j +k \mod t$
(the case $L(i,j,k) \equiv j -i +k \mod t$ can be done similarly).
By Property \ref{subcubecells}, $L(i_1,j_1,k_1) \equiv i_1 - j_1 + k_1 \mod t$. 
It is obvious that $(i_1,j_1,k_1)$ and $(i_2,j_2,k_2)$ are in the 
half transversal-set determined by $(i,j,k)$, and
according to Property \ref{twocellsinhalftransversalset}, 
we then have that $(i_1,j_1,k_1)$ and $(i_2,j_2,k_2)$ 
are in the same octant and $i_1-i_2 \equiv j_1-j_2 \equiv k_2 - k_1 \mod t$. 

Pick any cell $(i_3,j_3,k_3)$ in the half transversal-set $T_{h_1}$. Again,
it follows from
Property \ref{twocellsinhalftransversalset} that there is a subcube containing
$(i_3,j_3,k_3)$ and $(i_2,j_2,k_2)$. Thus, $(i_3,j_3,k_3) \in T_{h_2}$.
Therefore, $T_{h_1} = T_{h_2}$, a contradiction. We conclude that
two different half transversal-sets are disjoint.
\end{proof}

\begin{property}
\label{relatedhalftrans}
For any half transversal-set $T_{h_1}$ in the starting Latin cube of order $n$,
there is a unique half transversal-set $T_{h_2}$ 
such that:
\begin{itemize}

 \item for any cell $(i,j,k) \in T_{h_1}$, the half transversal-set determined 
by $(i,j,k)$ is $T_{h_2}$, and

\item for any cell $(i,j,k) \in T_{h_2}$, the half transversal-set determined 
by $(i,j,k)$ is $T_{h_1}$. 

\end{itemize}
Furthermore, $T_{h_1} \cup T_{h_2}$ contains $n$ different symbols. 
\end{property}

Two half transversal-sets $T_{h_1}$ and  $T_{h_2}$ as in the preceding
Property are called {\em related}.

\begin{proof}
Let $(i_1,j_1,k_1)$ be a cell that determines $T_{h_1}$. Consider a cell 
$(i_2,j_2,k_2) \in T_{h_1}$.
There is a unique half transversal-set $T_{h_2}$ 
determined by $(i_2,j_2,k_2)$, and
$(i_1,j_1,k_1) \in T_{h_2}$. 

\begin{itemize}
\item For any cell $(i,j,k) \in T_{h_1}$, we have that $(i_1,j_1,k_1)$ is in 
the half transversal-set  $T_h$ determined by $(i,j,k)$. 
Thus $(i_1,j_1,k_1) \in T_{h} \cap T_{h_2}$;
by Property \ref{twohalftransversalset}, this means that $T_h =T_{h_2}$.

\item For any cell $(i,j,k) \in T_{h_2}$, we have that $(i_2,j_2,k_2)$ is in 
the half transversal-set  $T_h$ determined by $(i,j,k)$. 
Thus $(i_2,j_2,k_2) \in T_{h} \cap T_{h_1}$, and
by Property \ref{twohalftransversalset}, this means that $T_h =T_{h_1}$.
 \end{itemize}
 
It is obvious that the set of symbols in $T_{h_1}$ and   the set of symbols 
in $T_{h_2}$ are disjoint.
Therefore $T_{h_1} \cup T_{h_2}$ contains $n$ different symbols. 
We conclude that for 
any half  transversal-set 
$T_{h_1}$,  there exists a unique half transversal-set $T_{h_2}$  such that 
$T_{h_1}$ and  $T_{h_2}$ are related. 
\end{proof}

\begin{definition}
\label{deftransversalset}
A \textit{transversal-set} of the starting Latin cube of order $n$ is a union of two related half transversal-sets.
\end{definition}

From the properties of half transversal-sets, we can immediately deduce the
Property \ref{protransversalset} for transversal-sets.
\begin{property}
\label{protransversalset}
Let  $\mathfrak{t}$ be a transversal-set in the starting Latin cube of order $n$, then
\begin{itemize}
\item $\mathfrak{t}$ is a set of $n$ cells and no two cells in $\mathfrak{t}$ contain the same symbol;
\item no two cells in $\mathfrak{t}$ are in the same row/column/file;
\item for any cell $(i_1,j_1,k_1)$ in $\mathfrak{t}$, $\mathfrak{t}$ contains all cells $(i_2,j_2,k_2)$
$(i_1 \neq i_2, j_1 \neq j_2, k_1 \neq k_2)$ such that $(i_1,j_1,k_1)$ and $(i_2,j_2,k_2)$ are two cells of a subcube.
\end{itemize}
Furthermore, any two distinct transversal-sets are disjoint.
\end{property}

\begin{property}
\label{halftranssymbol}
Let $(i,j,k)$ and $(i,j_x,k_x)$ be two cells in the same octant
of the starting Latin cube of order $n$ satisfying that $L(i,j,k)=L(i,j_x,k_x)$;
let $T_{h_1}$ and $T_{h_2}$ be two half transversal-sets containing 
$(i,j,k)$ and $(i,j_x,k_x)$, respectively.
If $(i',j',k')$ is a cell in $T_{h_1}$, then $T_{h_2}$ contains a cell $(i',j'_x,k'_x)$ such that $L(i',j',k')=L(i',j'_x,k'_x)$.
\end{property}
\begin{proof}
Assume that $L(i,j,k) \equiv i - j + k \mod t$ (the case $L(i,j,k) \equiv j -i +k \mod t$ is similar). 
By Property \ref{twocellsinhalftransversalset}, $(i,j,k)$ and $(i',j',k')$ are contained in the same octant $O_1$, 
$L(i',j',k') \equiv i' -j' + k' \mod t$ and $i - i' \equiv j - j' \equiv k' - k \mod t$. 
Since $(i,j_x,k_x)$ and $(i,j,k)$ are in the same octant, $(i,j_x,k_x) \in O_1$. Thus
$L(i, j_x,k_x) \equiv i - j_x + k_x \mod t$ and $i - j+ k \equiv i - j_x + k_x \mod t$, 
i.e. $k - j \equiv k_x - j_x \mod t$.
Now, there is only one cell $(i',j'_x,k'_x) \in O_1$ such that $j'_x \equiv j_x + i' - i \mod t$ and $k'_x \equiv k_x + i - i' \mod t$. 
Note that this choice of $j'_x$ and $k'_x$ implies that
 $i-i' \equiv j_x - j'_x \equiv k'_x - k_x \mod t$, which means that 
$(i',j'_x,k'_x) \in T_{h_2}$ by Property \ref{twocellsinhalftransversalset}.

We also have that $j_x - j'_x \equiv k'_x - k_x  \equiv i - i' \equiv j - j' \equiv k' - k \mod t$. 
Hence, $$L(i',j'_x,k'_x) \equiv i' - j'_x + k'_x \equiv i' - (j_x  + j' - j) + (k_x + k' - k)$$
$$\equiv i' - j' + k' + (k_x - j_x - k + j) 
\equiv i' - j' + k' \equiv L(i',j',k')  \mod t.$$ 
Since $(i',j',k')$ and $(i',j'_x,k'_x)$ are in the same octant $O_1$, this means
that $L(i',j',k')=L(i',j'_x,k'_x)$.
\end{proof}

\begin{property}
\label{rowblocknew}
Consider an arbitrary row $\{(i_1,j_1,k_1),\dots, (i_1,j_n,k_1)\}$
of the starting Latin cube $L$ of order $n$. 
For any $k_2$ $(i_2)$, there exists a unique $i_2$ $(k_2)$,
such that $L(i_1,x,k_1)=L(i_2,x,k_2)$ for every $x \in \{1,\dots,n\}$. 
\end{property}
\begin{proof}
We consider several cases:
\begin{itemize}

\item If $i_1, k_1 \leq t$, then for any $k_2 \leq t$, we can by Definition 
\ref{defSLatincube} 
choose a unique $i_2 \leq t$ such that $i_2  \equiv i_1 - k_1 + k_2\mod t$ to get
$$L(i_1,x,k_1)=x - i_1 + k_1 \mod t=x - i_2 + k_2 \mod t=L(i_2,x,k_2)$$ for every 
$x \in [t]$, 
and $$L(i_1,x,k_1)=(x - i_1 + k_1 \mod t) + t=(x - i_2 + k_2 \mod t) + t=L(i_2,x,k_2)$$ 
for every $x \in \{t+1,\dots,n\}$. 
Similarly, for any $i_2 \leq t$, we can choose a unique $k_2 \leq t$ such that $k_2  \equiv k_1 - i_1 +i_2 \mod t$. 

For any $k_2 >t$ ($i_2>t$), we can choose a unique $i_2 > t$ ($k_2>t$) such that $k_2 - i_2  \equiv k_1 - i_1 \mod t$. 

\item If $i_1, k_1>t$, then we can proceed exactly as in  the first case. That is, 
for any $k_2 \leq t$ ($i_2 \leq t$), we can choose a unique $i_2 \leq t$ ($k_2 \leq t$) such that $k_2 - i_2  \equiv k_1 - i_1 \mod t$;
for any $k_2 >t$ ($i_2>t$), we can choose a unique $i_2 > t$ ($k_2>t$) such that $k_2 - i_2  \equiv k_1 - i_1 \mod t$. 

\item For the two remaining cases, $i_1> t$, $k_1 \leq t$ and $i_1 \leq t$, $k_1 >t$, 
for any $k_2 >t$ ($i_2>t$), we can choose a unique $i_2 \leq t$ ($k_2 \leq t$) such that $k_2 + i_2  \equiv k_1 + i_1 \mod t$;
for any $k_2 \leq t$ ($i_2 \leq t$), we can choose a unique $i_2 > t$ ($k_2>t$) such that $k_2 + i_2  \equiv k_1 + i_1 \mod t$. 
\end{itemize}
\end{proof}

\begin{property}
\label{fileblocknew}
Consider an arbitrary file $\{(i_1,j_1,k_1),\dots, (i_1,j_1,k_n)\}$
of the starting Latin cube $L$ of order $n$. 
For any $i_2$ $(j_2)$, there exists a unique $j_2$ $(i_2)$,
such that $L(i_1,j_1,x)=L(i_2,j_2,x)$ for every $x \in \{1,\dots,n\}$. 
\end{property}

\begin{proof}
We consider several cases:
\begin{itemize}
\item If $i_1, j_1 \leq t$, then for any $i_2 \leq t$, we can
by Definition \ref{defSLatincube} 
choose a unique $j_2 \leq t$ such that $j_2  \equiv j_1 - i_1 + i_2\mod t$ to get
$$L(i_1,j_1,x)=j_1 - i_1 + x \mod t=j_2 - i_2 + x \mod t= L(i_2,j_2,x)$$
for every $x \in [t]$, and
$$L(i_1,j_1,x)=(i_1 - j_1 + x \mod t) +t=(i_2 - j_2 + x \mod t) +t= L(i_2,j_2,x)$$
for every $x \in \{t+1,\dots,n\}$. 
Similarly, for any $j_2 \leq t$, we can choose a unique $i_2 \leq t$ 
such that $i_2  \equiv i_1 - j_1 +j_2 \mod t$. 

For any $i_2 >t$ ($j_2>t$), we can choose a unique $j_2 > t$ ($i_2>t$) such that $i_1+i_2  \equiv j_1 + j_2 \mod t$
to get $$L(i_1,j_1,x)=j_1 - i_1 + x \mod t=i_2 - j_2 + x \mod t= L(i_2,j_2,x)$$ for every $x \in [t]$,
and $$L(i_1,j_1,x)=(i_1 - j_1 + x \mod t) +t=(j_2 - i_2 + x \mod t) +t= L(i_2,j_2,x)$$
for every $x \in \{t+1,\dots,n\}$.

\item Similarly, if $i_1, j_1>t$, then 
for any $i_2 \leq t$ ($j_2 \leq t$), we can choose a unique $j_2 \leq t$ ($i_2 \leq t$) such that $i_1+i_2 \equiv j_1+j_2\mod t$;
for any $i_2 >t$ ($j_2>t$), we can choose a unique $j_2 > t$ ($i_2>t$) such that $i_2 - j_2  \equiv i_1 - j_1 \mod t$. 

\item If $i_1> t$, $j_1 \leq t$, then 
for any $i_2 > t$ ($j_2 \leq t$), we can choose a unique $j_2 \leq t$ ($i_2 > t$) such that $i_2 - j_2  \equiv i_1 - j_1 \mod t$;
for any $i_2 \leq t$ ($j_2 > t$), we can choose a unique $j_2 > t$ ($i_2 \leq t$) such that $i_2 + j_2  \equiv i_1 + j_1 \mod t$. 

\item If $i_1 \leq t$, $j_1 >t$, then
for any $i_2 \leq t$ ($j_2 > t$), we can choose a unique $j_2 > t$ ($i_2 \leq t$) such that $i_2 - j_2  \equiv i_1 - j_1 \mod t$;
for any $i_2 >t$ ($j_2 \leq t$), we can choose a unique $j_2 \leq t$ ($i_2 > t$) such that $i_2 + j_2  \equiv i_1 + j_1 \mod t$. 
\end{itemize}
\end{proof}

\begin{definition}
Let $L$ be the starting Latin cube of order $n$.
A \textit{row block} of $L$ is a set of $n$ rows $R_{i,k}$ such that for every pair 
of rows $R_{i_1,k_1}=\{(i_1,j,k_1): j \in [n]\}$ and 
$R_{i_2,k_2} = \{(i_2,j,k_2): j \in [n]\}$ in this set, $L(i_1,x,k_1)=L(i_2,x,k_2)$ for every $x \in \{1,\dots,n\}$. 
\end{definition}

It is obvious that there are $n$ row blocks in total, and that
two distinct row blocks are disjoint.
\textit{File blocks} are defined similarly.

\begin{property}
\label{cellbelongsrowblock}
If 
$$\mathcal{C}=\{(i_1,j_1,k_1),  (i_1,j_2,k_1), (i_2,j_1,k_1), 
(i_2,j_2,k_1), (i_1,j_1,k_2), (i_1,j_2,k_2), (i_2,j_1,k_2), (i_2,j_2,k_2)\}$$
is a subcube in the starting Latin cube $L$,
then the two rows $R_{i_1,k_1}$ and $R_{i_2,k_2}$ are in the same row block,
as are also the two rows $R_{i_2, k_1}$ and $R_{i_1, k_2}$. 
\end{property}
Note that a  similar property holds for file blocks. Furthermore, this also implies that
every row block and file block contains exactly $n$ disjoint transversal-sets.
Unfortunately, no analogue of Property \ref{rowblocknew} holds for columns;
that is why we need to consider the notion of a {\em half column}.

\begin{definition}
\label{defhalfcolumn}
A \emph{half column} in the starting Latin cube of order $n=2t$ is a set of $t$ cells in a column such that the symbols used in this set
are $\mathcal{S}_1=\{1,2,\dots,t\}$ or $\mathcal{S}_2=\{t+1,\dots,n\}$.
\end{definition}

\begin{property}
\label{halftranscolumn}
Let $(i,j,k)$ and $(i_x,j,k)$ be two cells in a half column $C'_{j,k}$ of 
the starting Latin cube of order $n$, and
let $T_{h_1}$ and $T_{h_2}$ be the two 
half transversal-sets containing $(i,j,k)$ and $(i_x,j,k)$, respectively.
If $(i',j',k') \in T_{h_1}$ is a cell of a half column $C'_{j',k'}$, then $T_{h_2}$ contains a cell $(i'_x, j',k') \in C'_{j',k'}$.
\end{property}
\begin{proof}
Assume that $L(i,j,k) \equiv i - j + k \mod t$ 
(the case $L(i,j,k) \equiv j -i +k \mod t$ is similar). 
By Property \ref{twocellsinhalftransversalset}, $(i,j,k)$ and $(i',j',k')$ are contained in the same octant $O_1$, 
$i - i' \equiv j - j' \equiv k' - k \mod t$ and $C'_{j',k'}$ is contained in $O_1$.
Since $(i_x,j,k)$ and $(i,j,k)$ are in the same half column, 
they are  both contained in the octant $O_1$,
and $L(i_x,j,k) \equiv i_x - j + k \mod t$. Now, 
there is only one cell $(i'_x,j',k') \in C'_{j',k'}$ 
such that $i'_x \equiv i_x + i' - i \mod t$, that is, 
$i_x-i'_x \equiv j - j' \equiv k' - k \mod t$. Since $(i'_x,j',k') \in O_1$, it follows
from Property \ref{twocellsinhalftransversalset} that
$T_{h_2}$ contains the cell  $(i'_x,j',k') \in C'_{j',k'}$.
\end{proof}

\begin{definition}
\label{defhalfcolumnblock}
A \emph{first half column block} determined by $(i_1,j_1,k_1)$ in the starting Latin cube of order $n$ 
is a set containing every half column 
with the property that it has a cell
$(i_1,j_2,k_2)$ satisfying that
$$\mathcal{C}=\{(i_1,j_1,k_1),  (i_1,j_2,k_1), (i_2,j_1,k_1), 
(i_2,j_2,k_1), (i_1,j_1,k_2), (i_1,j_2,k_2), (i_2,j_1,k_2), (i_2,j_2,k_2)\}$$
is a subcube.
A \emph{second half column block} determined by $(i_1,j_1,k_1)$ 
in the starting Latin cube of order $n$ is a set containing every 
half column 
with the property that it has a cell
$(i_2,j_2,k_2)$ satisfying that $\mathcal{C}$ is a subcube.
\end{definition}

\begin{property}
\label{prophalfcolumnblock}
Every first  \emph{(}second\emph{)} half column block of the starting Latin cube  
of order $n=2t$ contains exactly $t$ half columns. 
Moreover, two distinct first  \emph{(}second\emph{)} half column blocks are disjoint and there are $4n$ first  \emph{(}second\emph{)} 
different half column blocks in total in the starting Latin cube.
\end{property}
\begin{proof}
Given any cell $(i_1,j_1,k_1)$, let $X_1$ be the set of cells $(i_1,j_2,k_2)$ and
$X_2$ be the set of cells $(i_2,j_2,k_2)$ satisfying that
$$\mathcal{C}=\{(i_1,j_1,k_1),  (i_1,j_2,k_1), (i_2,j_1,k_1), 
(i_2,j_2,k_1), (i_1,j_1,k_2), (i_1,j_2,k_2), (i_2,j_1,k_2), (i_2,j_2,k_2)\}$$
is a subcube. Since $|X_1|=|X_2|=t$ and the cells in $X_1$ ($X_2$) are contained in 
$t$ different file layers, no half column contains two cells from $X_1$ ($X_2$).
Hence, every first (second) half column block contains exactly $t$ half columns.

Let $\mathcal{C}_{ b_1}$ be the second half column block determined by $(i'_1,j'_1,k'_1)$
and $\mathcal{C}_{b_2}$ be the second half column block determined by $(i'_2,j'_2,k'_2)$.
Assume that the two cells $(i'_1,j'_1,k'_1)$ and $(i'_2,j'_2,k'_2)$ do not determine the same half column block, that is
$\mathcal{C}_{b_1} \neq \mathcal{C}_{b_2}$, and $\mathcal{C}_{b_1} \cap \mathcal{C}_{b_2} \neq \emptyset$.
Let $T_{h_1}$ and $T_{h_2}$ be the half transversal-set determined by $(i'_1,j'_1,k'_1)$ and $(i'_2,j'_2,k'_2)$, respectively.
Pick a half column $C'_{j,k} \in \mathcal{C}_{b_1} \cap \mathcal{C}_{b_2}$;
then in $C'_{j,k}$ there are cells
$(i,j,k) \in T_{h_1}$ and $(i_x,j,k) \in T_{h_2}$. Now, if 
$C'_{j',k'}$ is any half column from $\mathcal{C}_{b_1}$,
then there is a cell $(i',j',k') \in T_{h_1}$. 
By Property \ref{halftranscolumn}, $T_{h_2}$ contains a cell  
$(i'_x,j',k') \in C'_{j',k'}$; so, in fact, $C'_{j',k'} \in \mathcal{C}_{b_2}$.  
This implies that
 $\mathcal{C}_{b_1} = \mathcal{C}_{b_2}$, a contradiction. Hence, two distinct
second half column blocks are disjoint. 

Note that  $(i_1,j_2,k_2)$ and $(i_2,j_2,k_2)$ are always in two half columns of the 
same column $C_{j_2,k_2}$,
so given a first half column block $\mathcal{C}_h$, the set of half columns which are in the same column as
the half columns in $\mathcal{C}_h$ but do not belong to $\mathcal{C}_h$ is a second half column block. 
Thus for every first half column block there is a unique corresponding 
second half column block; therefore
two first half column blocks are disjoint. 
Moreover, since the starting Latin cube of order $n$ contains $n^3$ cells, 
every cell is contained in exactly one first (second) 
half column block, and every half column block contains $n^2/4$ cells,
there are $4n$ first (second) half column blocks in total.
\end{proof}

\begin{definition}
\label{defhalftransversalblock}
A \emph{first half transversal block} determined by $(i_1,j_1,k_1)$ in the starting Latin cube of order $n$ 
is a set containing every half transversal-set
with the property that it has a cell
$(i_1,j_2,k_2)$ satisfying that
$$\mathcal{C}=\{(i_1,j_1,k_1),  (i_1,j_2,k_1), (i_2,j_1,k_1), 
(i_2,j_2,k_1), (i_1,j_1,k_2), (i_1,j_2,k_2), (i_2,j_1,k_2), (i_2,j_2,k_2)\}.$$
is a subcube.
A \emph{second half transversal block} determined by $(i_1,j_1,k_1)$ 
in the starting Latin cube of order $n$ 
is a set containing every half transversal-set
with the property that it has a cell
$(i_2,j_1,k_1)$ satisfying that $\mathcal{C}$ is a subcube.
\end{definition}

\begin{property}
\label{prophalftransblock}
Every first \emph{(}second\emph{)} half transversal block of the starting Latin cube of 
order $n=2t$ contains exactly $t$ half transversal-sets. 
Two distinct first \emph{(}second\emph{)} half 
transversal blocks are disjoint and there are 
$4n$ first \emph{(}second\emph{)} half transversal blocks in total.
\end{property}

\begin{proof}
Given any cell $(i_1,j_1,k_1)$, let $X_1$ be the set of cells $(i_1,j_2,k_2)$
satisfying that $$\mathcal{C}=\{(i_1,j_1,k_1),  (i_1,j_2,k_1), (i_2,j_1,k_1), 
(i_2,j_2,k_1), (i_1,j_1,k_2), (i_1,j_2,k_2), (i_2,j_1,k_2), (i_2,j_2,k_2)\}$$
is a subcube. Since $|X_1|=t$ and $L(i_1,j_2,k_2)=L(i_1,j_1,k_1)$,
no half transversal-set contains two cells from $X_1$. Since distinct
half transversal-sets are disjoint, this means that
every first half transversal block contains exactly $t$ half transversal-sets.
By Property \ref{relatedhalftrans}, each half transversal-set 
containing a cell $(i_1,j_2,k_2)$
is related to exactly one half transversal-set containing a cell $(i_2,j_1,k_1)$,
and vice versa. This 
implies that
every second half transversal block contains exactly $t$ half transversal-sets, and
if two first half transversal blocks are disjoint,
then the related two second half transversal blocks
are also disjoint.

We will now prove that two distinct first half transversal blocks are disjoint. 
From Definition \ref{defhalftransversalblock}, it is easy to check that every first transversal block is contained in one octant.
Let $\mathcal{T}_{b_1}$ be the first half transversal block determined by $(i'_1,j'_1,k'_1)$
and $\mathcal{T}_{b_2}$ be the first half transversal block determined by $(i'_2,j'_2,k'_2)$.
Assume that the
two cells $(i'_1,j'_1,k'_1)$ and $(i'_2,j'_2,k'_2)$ do not determine the same half transversal block, i.e. $\mathcal{T}_{b_1} \neq \mathcal{T}_{b_2}$, 
and $\mathcal{T}_{b_1} \cap \mathcal{T}_{b_2} \neq \emptyset$.
Then $\mathcal{T}_{b_1}$ and $\mathcal{T}_{b_2}$ are in the same octant $O$. 

Pick a half transversal-set $T_{h_1} \in \mathcal{T}_{b_1} \cap \mathcal{T}_{b_2}$.
Then $T_{h_1}$ contains a cell $(i'_1, j,k)$ in row layer $i'_1$ and a cell $(i'_2, j',k')$ in row layer $i'_2$ 
such that $L(i'_1, j,k)=L(i'_1,j'_1,k'_1)$ and $L(i'_2, j',k')=L(i'_2,j'_2,k'_2)$.
Now, for any half transversal-set $T_{h_2} \in \mathcal{T}_{b_1}$, 
there is a cell $(i'_1,j_x,k_x) \in T_{h_2}$
such that $$L(i'_1,j_x,k_x)=L(i'_1,j'_1,k'_1)=L(i'_1, j,k).$$ 
Since $(i'_1, j,k), (i'_1,j_x,k_x) \in O$, by Property \ref{halftranssymbol}, $T_{h_2}$ contains a cell 
$(i'_2,j'_x,k'_x) \in O$ such that $$L(i'_2,j'_x,k'_x)=L(i'_2, j',k')=L(i'_2,j'_2,k'_2).$$
Furthermore, in the row layer $i'_2$, $(i'_2,j'_x,k'_x)$ and $(i'_2,j'_2,k'_2)$ are in two opposite quandrants.
Thus by Property \ref{cellsinoppositequadrants}, there is a subcube containing 
$(i'_2,j'_x,k'_x)$ and $(i'_2,j'_2,k'_2)$.
Therefore $T_{h_2} \in \mathcal{T}_{b_2}$. 
This implies that $\mathcal{T}_{b_1} = \mathcal{T}_{b_2}$, a contradiction. 
Hence, two distinct first half transversal blocks are disjoint,
as are also distinct second half transversal blocks.

Furthermore, since the starting Latin cube of order $n$ contains $n^3$ cells, 
every cell is contained in exactly one first (second) 
half transversal block, and every half transversal block contains $n^2/4$ cells,
there are $4n$ first (second) half transversal blocks in total.
\end{proof}

\begin{property}
\label{blockdetermine}
\begin{itemize}

\item[(i)] The set of cells that determine the same first \emph{(}second\emph{)} 
half column block form a first half transversal block.
\item[(ii)] The set of cells that determine the same first \emph{(}second\emph{)} half transversal block form a first half column block.

\end{itemize}
\end{property}
\begin{proof}
Let $(i_1,j_1,k_1)$ and $(i_2,j_2,k_2)$ be any two cells that determine the first half column block $\mathcal{C}_b$,
then $L(i_1,j_1,k_1)$ and $L(i_2,j_2,k_2)$ belong to the same 
subset $S$ of symbols,
i.e., $S=S_1$ or $S=S_2$.
By Definition \ref{defhalfcolumnblock}, the half columns in $\mathcal{C}_b$
contains 
two sets $X_1$ and $X_2$ both consisting of $t$ cells, where $X_1$ contains
cells from row layer $i_1$ 
and $X_2$ contains cells from row layer $i_2$ , and where for any $(i_1,j'_1,k'_1) \in X_1$ we have $L(i_1,j'_1,k'_1)=L(i_1,j_1,k_1)$ 
and for any $(i_2,j'_2,k'_2) \in X_2$ we have $L(i_2,j'_2,k'_2)=L(i_2,j_2,k_2)$.
Let $(i_1,j,k)$ be any cell in $X_1$; then, by Definition \ref{defhalftransversalblock},
 $(i_1,j_1,k_1)$ 
is in the first half transversal block $\mathcal{T}_b$ determined by $(i_1,j,k)$. 

\begin{itemize}

\item If $L(i_1,j_1,k_1)=L(i_2,j_2,k_2)$, then $X_1=X_2$,
since the half columns in $C_b$ only contains $t$ cells
all of which contain the same symbol $L(i_1,j_1,k_1)$. Thus $(i_1,j,k) \in X_2$ and
$(i_2,j_2,k_2)$ are in the first half transversal block $\mathcal{T}_b$ determined by $(i_1,j,k)$. 

\item If $L(i_1,j_1,k_1) \neq L(i_2,j_2,k_2)$, then let $T_{h_1}$ be the half transversal-set containing $(i_1,j_1,k_1)$.
Then $T_{h_1} \in \mathcal{T}_{b}$ and the set of symbols in $T_{h_1}$ is $S$. Therefore, we can pick a cell
$(i_3,j_3,k_3) \in T_{h_1}$ such that $L(i_3,j_3,k_3)=L(i_2,j_2,k_2)$. If $(i_3,j_3,k_3)=(i_2,j_2,k_2)$, we are done.
Otherwise, let $T_{h_2}$ be the half transversal-set
related to $T_{h_1}$. 
By Property \ref{relatedhalftrans}, $T_{h_2}$ is 
determined by both $(i_3,j_3,k_3)$ and $(i_1,j_1,k_1)$.
It follows that these two cells determine the same second half column block, 
which means that they also determine the 
same first half column block. Now, both $(i_3,j_3,k_3)$ and $(i_2,j_2,k_2)$ determine the 
first half column block $\mathcal{C}_b$, and $L(i_3,j_3,k_3)=L(i_2,j_2,k_2)$, 
so as in the previous case, it follows that
$(i_2,j_2,k_2)$ and $(i_3,j_3,k_3)$ are in the same first transversal block, 
that is, $(i_2,j_2,k_2) \in \mathcal{T}_b$.
\end{itemize}

We conclude that the set of cells that determine the same first half column block 
are all contained in the same first half transversal block.
It follows immediately that the set of cells that determine the same 
second half column block are all contained in the same first half transversal block.
By proceeding similarly, we can prove that the set of cells that determine 
the same first (second) half transversal block are all contained in the same 
first half column block.
\end{proof}

\begin{definition}
A \textit{symbol-set} of the starting Latin cube $L$ of order $n$ in a row layer (or column layer or file layer)
 is a set of $n$ cells such that all these cells contain the same symbol. 
\end{definition}

\begin{definition}
A \textit{symbol block} of the starting Latin cube $L$ of order $n$ is a set of $n^2$ cells such that all these cells contain the same symbol. 
\end{definition}

\begin{property}
\label{Scolumnblock}
Let $L$ be the starting Latin cube of order $n$, $\mathfrak{b}$ an 
arbitrary symbol in $L$, and
$B_1$ be the set of cells of $L$ in the first 
column layer which contain symbol $\mathfrak{b}$.
For any column layer $j$, the set of cells $B_j$ of $L$ 
in column layer $j$ which have the 
same row and file coordinates as the cells in $B_1$ all contain the same symbol. 
\end{property}
\begin{proof}
Assume $(i_1,j,k_1) \in B_j$, 
$L(i_1,j,k_1)=x \equiv i_1 - j + k_1 \mod t$
(the case $L(i_1,j,k_1)\equiv j - i_1 + k_1 \mod t$ is similar).
Consider an arbitrary cell $(i_2,j,k_2) \in B_j$.
By definition, there are two cells $(i_1,1,k_1)$ and $(i_2,1,k_2)$ such that 
$L(i_1,1,k_1) = L(i_2,1,k_2) = \mathfrak{b}$. 
Since $L(i_1,j,k_1) \equiv i_1 - j + k_1 \mod t$, by Definition \ref{defSLatincube}
we have $L(i_1,1,k_1) \equiv i_1 - 1 + k_1 \mod t$. 
It follows from Property \ref{cellproperty}
that regardless of whether $(i_1,1,k_1)$ and $(i_2,1,k_2)$ are in the same quadrant 
or in two opposite quadrants, it holds that 
$L(i_2,1,k_2) \equiv i_2 - 1 + k_2 \mod t$.
This yields that $L(i_2,j,k_2) \equiv i_2 - j + k_2 \mod t$.
In consequence, $i_1 - 1 + k_1 \equiv  i_2 - 1 + k_2 \mod t$, which means that
 $i_1 - j + k_1 \equiv  i_2 - j + k_2 \mod t$, and so 
 $L(i_2,j,k_2) \equiv L(i_1,j,k_1) \mod t$. Since the set of symbols in a quadrant is $S_1$ or $S_2$ and the
 set of symbols in two opposite 
quadrants are the same, we deduce that  $L(i_2,j,k_2) = L(i_1,j,k_1)=x$.
 
We conclude that all cells in $B_j$ contain the same symbol.
\end{proof}

\begin{property}
\label{Sfileblock}
Let $L$ be the starting Latin cube of order $n$, 
$\mathfrak{b}$ an arbitrary symbol in $L$, and
$B_1$ be the set of cells of 
$L$ in the first file layer which contain $\mathfrak{b}$.
For any file layer $k$, the set of cells $B_k$ of 
$L$ in file layer $k$ which have the 
same row and column coordinates as the cells in $B_1$ all contain the same symbol. 
\end{property}

\begin{proof}
Assume $(i_1,j_1,k) \in B_k$, $L(i_1,j_1,k)=x \equiv i_1 - j_1 + k \mod t$
(the case $L(i_1,j,k_1)\equiv j_1 - i_1 + k \mod t$ is similar).
Consider an arbitrary cell $(i_2,j_2,k) \in B_k$.
By definition, there are two cells $(i_1,j_1,1)$ and $(i_2,j_2,1)$ such that 
$L(i_1,j_1,1) = L(i_2,j_2,1) = \mathfrak{b}$. We consider two cases.

\begin{itemize}
\item If $(i_1,j_1,1)$ and $(i_2,j_2,1)$ are in the same quadrant in the file layer $1$,
then $i_1 - j_1 + 1 \equiv i_2 - j_2 +1 \mod t$ or $j_1 - i_1 +1 \equiv j_2 - i_2 +1 \mod t$;
in both cases $i_1 - j_1 \equiv i_2 - j_2 \mod t$. 
Since $(i_1,j_1,k)$ and $(i_2,j_2,k)$ are 
in the same quadrant in the file layer $k$ and
 $L(i_1,j_1,k) \equiv i_1 - j_1 + k \mod t$, it holds that
$L(i_2,j_2,k) \equiv i_2 - j_2 + k \equiv i_1 - j_1 + k \mod t$, that is,
$L(i_2,j_2,k) \equiv L(i_1,j_1,k) \mod t$. 
Since the set of symbols in a quadrant is $S_1$ or $S_2$,
$L(i_2,j_2,k) = L(i_1,j_1,k) =x$.

\item If $(i_1,j_1,1)$ and $(i_2,j_2,1)$ are in two opposite quadrants 
in the file layer $1$,
then Property \ref{cellproperty} implies that $i_1 - j_1 + 1 \equiv j_2 - i_2 +1 \mod t$ or $j_1 - i_1 +1 \equiv i_2 - j_2 +1 \mod t$;
in both cases it holds that $i_1 + i_2 \equiv j_1 + j_2 \mod t$. 
Since $(i_1,j_1,k)$ and $(i_2,j_2,k)$ are 
in  two opposite quadrants in the file layer $k$, and
$L(i_1,j_1,k) \equiv i_1 - j_1 + k \mod t$, 
we have that $L(i_2,j_2,k) \equiv j_2 - i_2 + k \mod t$,
and since $j_2 - i_2 + k \equiv i_1 - j_1 + k \mod t$, 
$L(i_2,j_2,k) \equiv L(i_1,j_1,k) \mod t$. 
Since the set of symbols in two opposite quadrants are the same,
$L(i_2,j_2,k) = L(i_1,j_1,k) =x$. 
\end{itemize}
We conclude that all cells in $B_k$ contain the same symbol.
\end{proof}

Based on Property \ref{Scolumnblock}, we make the following definition.
\begin{definition}
A \textit{symbol-column block} of the 
starting Latin cube $L$ is a set $\mathfrak{s}$  of $n^2$ cells 
satisfying that
\begin{itemize}
	
	\item all cells of $\mathfrak{s}$ that are in the same column layer
	contain the same symbol, and
	
	\item for every cell of $\mathfrak{s}$, there are $n-1$ other cells that have
	the same row and file coordinate.

\end{itemize}
\textit{Symbol-file blocks} are defined similarly. 
\end{definition}
Unfortunately, no analogue of Property \ref{Scolumnblock} holds for row layers. 
Therefore we need to consider
{\em half symbol-sets} in row layers and {\em half symbol-row blocks}.

\begin{definition}
A \textit{half symbol-set} of the starting Latin cube $L$ of order $n=2t$
 is a set of $t$ cells in a quadrant of a row layer such that all these cells contain the same symbol. 
\end{definition}

\begin{definition}
A \emph{first half symbol-row block}  determined by $(i_1,j_1,k_1)$ in the starting Latin cube of order $n$ 
is a set containing every half symbol-set 
with the property that it has a cell
$(i_2,j_1,k_1)$ satisfying that
$$\mathcal{C}=\{(i_1,j_1,k_1),  (i_1,j_2,k_1), (i_2,j_1,k_1), 
(i_2,j_2,k_1), (i_1,j_1,k_2), (i_1,j_2,k_2), (i_2,j_1,k_2), (i_2,j_2,k_2)\}.$$
is a subcube.
A \emph{second half symbol-row block} determined by $(i_1,j_1,k_1)$  in the starting Latin cube of order $n$ 
is a set containing every half symbol-set 
with the property that it has a cell
$(i_2,j_2,k_2)$ satisfying that $\mathcal{C}$ is a subcube.
\end{definition}

\begin{property}
\label{propsymbolblock}
\begin{itemize}

\item[(i)] 
Every first \emph{(}second\emph{)} half symbol-row block of the starting Latin cube of 
order $n=2t$ contains exactly $t$ half symbol-sets. 

\item[(ii)]Two distinct first \emph{(}second\emph{)} half symbol-row blocks are disjoint, and there are $4n$ first \emph{(}second\emph{)} half symbol-row blocks.

\item[(iii)] The set of cells that determine the same
 first \emph{(}second\emph{)} half 
symbol-row block form a second half column block.
\end{itemize}
\end{property}

The proof of Property \ref{propsymbolblock} is similar to the proofs of 
Properties \ref{prophalfcolumnblock}, \ref{prophalftransblock} and \ref{blockdetermine};
the details are omitted.

The following simple observation enables us to permute layers and symbols
in a Latin cube.

\begin{property}
If $L$ is a Latin cube, then
the cube obtained by permuting the row layers, the column layers, 
the file layers and/or the symbols of $L$ is a Latin cube.
\end{property}

For starting Latin cubes an even stronger property holds.
If a Latin cube $L'$ is obtained from another Latin cube $L$ by permuting
row/column/file layers and/or symbols of $L$, then we say
that $L$ and $L'$ are {\em isotopic}.

Note that for a Latin cube $L'$ that is isotopic to the starting Latin cube $L$, 
we can define octants, quadrants, transversal-sets, row blocks, file blocks,
symbol-column blocks, symbol-file blocks, first (second) half column blocks, 
first (second) half transversal blocks, first (second) half symbol-row blocks of $L'$ similarly to the definitions for $L$.
Moreover, Properties \ref{tsubcube},  \ref{cellsinoppositequadrants}, 
\ref{symbolsetofhalftransversalset}, 
\ref{twohalftransversalset}, 
\ref{relatedhalftrans}, 
\ref{protransversalset}, 
\ref{halftranssymbol}, 
\ref{rowblocknew},  \ref{fileblocknew},  
\ref{cellbelongsrowblock}, 
\ref{halftranscolumn},  
\ref{prophalfcolumnblock}, 
\ref{prophalftransblock},  
\ref{blockdetermine},  
\ref{Scolumnblock},  \ref{Sfileblock},  \ref{propsymbolblock} hold for $L'$.


\bigskip

Given an $n\times n \times n$ cube $A$ where each cell contains a subset of the
symbols in $\{1,\dots,n\}$, and a Latin cube $L$ of order $n$
that does not avoid $A$, we say that those cells $(i,j,k)$ of $L$ where 
$L(i,j,k) \in A(i,j,k)$ are \textit{conflict cells of $L$ with $A$} 
(or simply {\em conflicts} of $L$).
An \textit{allowed subcube} of $L$ is a subcube 
$$\mathcal{C}=\{(i_1,j_1,k_1),  (i_1,j_2,k_1), (i_2,j_1,k_1), (i_2,j_2,k_1), 
(i_1,j_1,k_2),  
(i_1,j_2,k_2),  (i_2,j_1,k_2),   (i_2,j_2,k_2)\}$$ in $L$ such that 
swapping on $\mathcal{C}$ yields a Latin cube
$L'$ where none of $(i_1,j_1,k_1), $ $ (i_1,j_2,k_1), $ $ (i_2,j_1,k_1), $ $ (i_2,j_2,k_1),$ $ (i_1,j_1,k_2), $ 
$(i_1,j_2,k_2), $ $ (i_2,j_1,k_2), $ $  (i_2,j_2,k_2)$ is a conflict.

\section{Proof of the main theorem}



Now that all our preparatory lemmas and definitions are in place, we are
ready to prove Theorem \ref{maintheorem}.
Our basic proof strategy is similar to the one in 
\cite{JohanKlasLan}, in fact the only substantial difference
is that we employ the extended machinery
developed in Section 2 in place of the tools from \cite{JohanKlasLan}.

Our starting point in the proof is the starting Latin cube; we 
permute its row layers, column layers, file layers and symbols 
so that the resulting Latin cube does not have 
too many conflicts with a given $(m,m,m,m)$-cube $A$. After that, we find a set of allowed subcubes such that each conflict belongs to one of them, with no two of the subcubes intersecting, and swap on those subcubes.

The proof of Theorem \ref{maintheorem} involves a number of parameters:
$$\alpha, \gamma, \kappa, \epsilon, \theta,$$
and a number of inequalities that they must satisfy. For the reader's convenience,
explicit choices for which the proof holds are presented here:
\begin{equation}
\label{eq:param}
\alpha = 1/2-38 \times 2^{-27}, \gamma=2^{-27}, 
\kappa= 6 \times 2^{-27}, \epsilon = 2^{-6}, \theta=2^{-13}.
\end{equation}

By examples of unavoidable
$(\lfloor {\frac{n}{3}} \rfloor+1, 
\lfloor {\frac{n}{3}} \rfloor+1, \lfloor {\frac{n}{3}} \rfloor+1)$-arrays
in \cite{CulterOhman} (see also \cite{JohanKlasLan}), 
the value of $\gamma$ for which Theorem \ref{maintheorem}
holds cannot exceed $\frac{1}{3}$.
Thus, since the numerical value of $\gamma$ for which the theorem holds
is not anywhere near
what we expect to be optimal, we have not put an effort into choosing optimal values
for these parameters. 
Moreover, for simplicity of notation,
we shall omit floor and ceiling signs whenever these are not crucial.

We shall establish that our main theorem holds by
proving two lemmas.

\begin{lemma}
Let $\alpha, \gamma, \kappa$ be constants and $n=2t$ such that  

$$ \Big(7n^2 \dfrac{(\gamma n)^{\kappa n}}{(\kappa n)!} +3n^3 \dfrac{{(2\gamma n)^{(1/2-\alpha-2\gamma)n/3}}}{((1/2-\alpha-2\gamma)n/3)!}\Big) < 1.$$
For any $(\gamma n, \gamma n, \gamma n, \gamma n)$-cube $A$ of order $n$ there is a quadruple of permutations $\sigma=(\tau_1, \tau_2, \tau_3, \tau_4)$
of the row layers, the column layers, the file layers and the symbols of the starting Latin cube $L_0$ of order $n$,
respectively, 
such that applying $\sigma$ to $L_0$, we obtain a Latin cube $L$ satisfying
the following:
\begin{itemize}
\item[(a)] No row in $L$ contains more than $\kappa n$ conflicts with $A$.
\item[(b)] No column in $L$ contains more than $\kappa n$ conflicts with $A$.
\item[(c)] No file in $L$ contains more than $\kappa n$ conflicts with $A$.
\item[(d)] No symbol-set in $L$ contains more than $\kappa n$ conflicts with $A$.
\item[(e)] No transversal-set in $L$ contains more than $\kappa n$ conflicts with $A$.
\item[(f)] Each cell of $L$ belongs to at least $\alpha n$ allowed subcubes.
\end{itemize}
\end{lemma}

The proof of this lemma is virtually identical to the proof of Lemma
22 in \cite{JohanKlasLan}. For completeness, we provide a brief sketch of the proof,
for details, see \cite{JohanKlasLan}.

\begin{proof} (Sketch.)
Let $X_a$, $X_b$, $X_c$, $X_d$, $X_e$ and $X_f$ be the number of permutations which do not fulfill the conditions
$(a)$, $(b)$, $(c)$, $(d)$, $(e)$ and $(f)$, respectively. Let $X$ be the number of permutations satisfying the 
six conditions $(a)$, $(b)$, $(c)$, $(d)$, $(e)$ and $(f)$. There are $(n!)^4$ ways to permute the row layers,
the column layers, the file layers and the symbols, so we have
$$X \geq (n!)^4 - X_a - X_b - X_c - X_d - X_e - X_f.$$
So for proving the lemma, it suffices to show that $X$ is greater than $0$.
To give the reader an idea of the flavour of the proof,
we shall briefly demonstrate how $X_a$ can be estimated.

Assume that for any fixed permutation 
$(\tau_1,\tau_3,\tau_4)$ of the row layers,
the file layers and the symbols, 
at most $N_a$ choices of a permutation $\tau_2$ of the column layers yield a quadruple
$(\tau_1, \tau_2, \tau_3,\tau_4)$ of permutations 
that break condition $(a)$; so $X_a \leq n!n!n!N_a$.

Let $R$ be a fixed row chosen arbitrarily; 
we count the number of ways a permutation $\tau_2$ of the column layers 
can be constructed so that $(a)$ does not hold on row $R$.
Let $S$ be a set of size $\kappa n$ of column layers of $A$;
there are $n \choose \kappa n$ ways to choose $S$. 
In order to have a conflict at cell $(i,j,k)$ of $R$, the column layers 
should be permuted
in such a way that in the resulting Latin cube $L$, $L(i,j,k) \in A(i,j,k)$. 
Since $|A(i,j,k)| \leq \gamma n$, there are at most $(\gamma n)^{\kappa n}$ 
ways of choosing which column 
layers of $B$ are mapped by $\tau_2$ to column layers in $S$ so that 
all cells on row $R$ that are in $S$ are conflicts. The rest of the column 
layers can be arranged in any of
the $(n-\kappa n)!$ possible ways. In total this gives at most
$${n \choose \kappa n}(\gamma n)^{\kappa n}(n-\kappa n)! 
= \dfrac{n!(\gamma n)^{\kappa n}}{(\kappa n)!}$$
permutations $\tau_2$ that do not satisfy condition $(a)$ on row $R$.
There are $n^2$ rows in $B$, so we have 
$$N_a \leq n^2 \dfrac{n!(\gamma n)^{\kappa n}}{(\kappa n)!}$$
and $$X_a \leq n!n!n!N_a \leq n^2(n!)^4\dfrac{(\gamma n)^{\kappa n}}{(\kappa n)!}.$$

Using similar estimates for $X_b, \dots, X_e$, one may then deduce that
 $$X_a + X_b + X_c +X_d + X_e  \leq 7n^2(n!)^4\dfrac{(\gamma n)^{\kappa n}}{(\kappa n)!};$$
for details, see the proof of  Lemma 22 in \cite{JohanKlasLan}.

The argument for estimating $X_f$ uses similar
counting arguments but is more involved; here as well the details can be found in the proof of
Lemma 22 in \cite{JohanKlasLan}, and we omit them here.
It turns out that
$$X_f \leq 
3n^3 (n!)^4 \dfrac{{(2\gamma n)^{(1/2-\alpha-2\gamma)n/3}}}
{((1/2-\alpha-2\gamma)n/3)!}$$
Summing up, we conclude that
\vspace{1.3cm}
$$X \geq (n!)^4 - 7n^2 (n!)^4\dfrac{(\gamma n)^{\kappa n}}{(\kappa n)!} - 3n^3 (n!)^4 \dfrac{{(2\gamma n)^{(1/2-\alpha-2\gamma)n/3}}}{((1/2-\alpha-2\gamma)n/3)!}$$
$$\geq (n!)^4 \Big(1 - 7n^2 \dfrac{(\gamma n)^{\kappa n}}{(\kappa n)!} - 3n^3 \dfrac{{(2\gamma n)^{(1/2-\alpha-2\gamma)n/3}}}{((1/2-\alpha-2\gamma)n/3)!}\Big)$$
By \eqref{eq:param}, $X$ is strictly greater than $0$, provided
that $n$ is large enough.
\end{proof}

\begin{lemma}
Let $L$ be a Latin cube that is isotopic to a starting Latin cube,
and let $A$ be an $(m,m,m,m)$-cube; both of order $n=2t$.
Furthermore, let $\alpha, \gamma, \kappa, \theta, \epsilon$ be 
constants such that $\epsilon n \geq 3$ and

$$\alpha n-21\kappa n -7\epsilon n-\dfrac{90\kappa}{\epsilon}n-\dfrac{24\theta}{\epsilon} n-\dfrac{544\kappa}{\theta} n-25>0.$$

If $L$ has the following properties:
\begin{itemize}
\item[(a)] no row in $L$ contains more than $\kappa n$ conflicts with $A$;
\item[(b)] no column in $L$ contains more than $\kappa n$ conflicts with $A$;
\item[(c)] no file in $L$ contains more than $\kappa n$ conflicts with $A$;
\item[(d)] no symbol-set in $L$ contains more than $\kappa n$ conflicts with $A$;
\item[(e)] no transversal-set in $L$ contains more than $\kappa n$ conflicts 
with $A$;
\item[(f)] each cell of $L$ belongs to at least $\alpha n$ allowed subcubes;
\end{itemize}
then there is a set of disjoint allowed subcubes such that each conflict 
of $L$ belongs to one of them. Thus, by performing a number 
of swaps on subcubes in $L$, 
we obtain a Latin cube $L'$ that avoids $A$.
\end{lemma}

\begin{proof}
For constructing $L'$ from $L$, we will perform a number of swaps on subcubes,
and we shall refer to this procedure as
\textit{$S$-swap}.  We are going to construct a set $S$ of disjoint allowed 
subcubes such that each conflict of $L$ with $A$ belongs to one of them. 
A cell that belongs to a subcube in $S$ is called \textit{used} in $S$-swap. 
Since no row in $L$ contains more than $\kappa n$ conflicts with $A$, there are at most $\kappa n^3$ conflicts in $L$, which implies 
that the total number of cells that are used in 
$S$-swap is at most $8 \kappa n^3$. 
Furthermore, every block and layer contains at most $\kappa n^2$ conflicts,
and
every half block contains at most $\kappa n^2/2$ conflicts.

A row layer, a column layer, a file layer, a row block, a file block, a symbol block, a symbol-column block, 
or a symbol-file block is \textit{overloaded} if such a layer or block
contains at least $\theta n^2$ cells that are used in $S$-swap; note that no more than $\dfrac{8\kappa n^3}{\theta n^2}=\dfrac{8\kappa}{\theta} n$ 
layers or blocks of each type are overloaded. 

A first (second) half column block, a first (second) half transversal block, or a first (second) half symbol-row block
is \textit{overloaded} if such a half block contains at least $\theta n^2 / 2$ cells that are used in $S$-swap;
note that no more than $\dfrac{8\kappa n^3}{\theta n^2 / 2}=\dfrac{16\kappa}{\theta} n$ 
half blocks of each type are overloaded. 

A row, a column, a file, a transversal-set, 
or a symbol-set is \textit{overloaded} if 
this row, column, file, transversal-set or symbol-set contains 
at least $\epsilon n$ cells that are used in $S$-swap.
A half column, a half transversal-set, or a half symbol-set is \textit{overloaded} 
if this half column, half transversal-set, or half symbol-set
contains at least $\epsilon n / 2$ cells that are used in $S$-swap.

Using these facts, let us now construct our set $S$ by steps; at each step we consider a conflict cell $(i_1,j_1,k_1)$ and include an allowed subcube 
containing $(i_1,j_1,k_1)$ in $S$. Initially, the set $S$ is empty.

So let us consider a conflict cell $(i_1,j_1,k_1)$ in $L$; there are at least $\alpha n$ allowed subcubes containing $(i_1,j_1,k_1)$. 
We choose an allowed subcube 
$$\mathcal{C}=\{(i_1,j_1,k_1),  (i_1,j_2,k_1),  (i_2,j_1,k_1), 
(i_2,j_2,k_1),  (i_1,j_1,k_2), (i_1,j_2,k_2), (i_2,j_1,k_2), 
(i_2,j_2,k_2)\}$$ that satisfies the following:
\begin{itemize}

\item[(1)] The row layer $i_2$, the column layer $j_2$, the file layer $k_2$, the row block containing the row $R_{i_2,k_1}$,
the file block containing the file $F_{i_1,j_2}$, the symbol-column block containing the two cells 
$(i_2,j_1,k_1)$ and $(i_1,j_1,k_2)$, the symbol-file block containing 
the two cells $(i_1,j_2,k_1)$ and $(i_2,j_1,k_1)$, and
the symbol block containing symbol $L(i_1,j_2,k_1)$ are not overloaded. 
This eliminates at most $\dfrac{8 \times 8\kappa}{\theta} n=\dfrac{64\kappa}{\theta} n$ choices.

The four first (second) half column blocks each of which contains 
one cell from
the set
$\mathcal{X}=\{(i_1,j_2,k_1), $ $(i_2,j_2,k_1), $ $(i_1,j_1,k_2), $ $(i_2,j_1,k_2)\}$ are not overloaded.
This eliminates at most $\dfrac{2 \times 4 \times 16\kappa}{\theta} n=\dfrac{128\kappa}{\theta} n$ choices.
Similarly, we need that the four first (second) half transversal blocks 
each of which contains 
one cell from
the set $\mathcal{X}$, the four first (second) half symbol-row blocks 
each of which contains 
one cell from
the set $\mathcal{X}$ are not overloaded.
This eliminates at most $\dfrac{256\kappa}{\theta} n$ choices.
The first half column block containing $(i_2,j_2,k_2)$ and the second half column block containing $(i_1,j_2,k_2)$,
the first half transversal block containing $(i_2,j_1,k_1)$ and the second half transversal block containing $(i_1,j_2,k_2)$,
the first half symbol-row block containing $(i_2,j_2,k_2)$ and the second half symbol-row block containing $(i_2,j_1,k_1)$
are not overloaded. This eliminates at most $\dfrac{6 \times 16\kappa}{\theta} n=\dfrac{96\kappa}{\theta} n$ choices.

With this strategy for including subcubes in $S$, 
after completing the construction of $S$, every layer (or block) contains at most 
$4\kappa n^2 + (\theta n^2-1) +4$ cells that are used in $S$-swap.
Consider any first half column block $\mathcal{B}_1$, and 
let $\mathcal{B}_2$ be the second
half column block that corresponds uniquely to $\mathcal{B}_1$, that is,
$\mathcal{B}_2 \cup \mathcal{B}_1$
together form a set of $t$ columns, and let $\mathcal{T}_b$ be the first transversal block that determines
$\mathcal{B}_1$ (such a first transversal block exists
by Property \ref{blockdetermine}). Now,  a cell $(i,j,k)$ of $\mathcal{B}_1$
is used in $S$-swap if
$$\mathcal{C}=\{(i,j,k),  (i,j_x,k),  (i_x,j,k), (i_x,j_x,k),  (i,j,k_x), (i,j_x,k_x), (i_x,j,k_x), (i_x,j_x,k_x)\}$$
is a subcube and:

\begin{itemize}

\item $(i,j,k)$ is a conflict; there are at most $\kappa n^2 /2$ conflicts in 
$\mathcal{B}_1$; or

\item $(i_x,j,k)$ is a conflict cell of $\mathcal{B}_2$; there 
are at most $\kappa n^2 /2$ conflicts in $\mathcal{B}_2$; or

\item $(i,j_x,k_x)$ is a conflict cell of $\mathcal{T}_b$;
there are at most $\kappa n^2 /2$ conflicts in $\mathcal{T}_b$; or

\item one of the five cells $\{(i,j_x,k), (i_x,j_x,k),  (i,j,k_x), (i_x,j,k_x), (i_x,j_x,k_x)\}$ is a conflict.
\end{itemize}

Hence, after completing the construction of $S$, 
every first half column block contains at most 
$3 \times (\kappa n^2 /2)+ (\theta n^2 /2 -1) +1$ cells that are used in $S$-swap.
Similarly, every second half column block, every first (second) transversal block, every first (second) symbol-row
block contains at most $3\kappa n^2 /2+ \theta n^2 /2$ cells that are used in $S$-swap.

We conclude that the number 
of overloaded rows (overloaded columns, overloaded files, 
overloaded transversal-sets or overloaded symbol-sets) in each layer (or block)  
is at most $\dfrac{4\kappa n^2 + \theta n^2 +3}{\epsilon n} \leq \dfrac{4\kappa + \theta}{\epsilon} n+1$. 
We also have that the number of overloaded half columns, half transversal-sets, half symbols-sets in each half block
is at most $\dfrac{3\kappa n^2/2 + \theta n^2 /2}{\epsilon n /2} =\dfrac{3\kappa + \theta}{\epsilon}n$ 
conflicts.

Note that here the statement ``each symbol block contains at most 
$\dfrac{4\kappa + \theta}{\epsilon} n+1$ overloaded symbol-sets''
is to be taken with respect to either row layers, column layers
or file layers, i.e., when we consider the $n$ different symbol sets of a given symbol
block belonging  to $n$ different row layers
(or $n$ different column layers or $n$ different file layers), 
the number of overloaded such symbol-sets is at most
$\dfrac{4\kappa + \theta}{\epsilon} n+1$.

\item[(2)] Some rows, columns, files, transversal-sets, symbol-sets, 
half columns, half symbol-sets, half transversal-sets are not overloaded 
according to the following:
\begin{itemize}
\item[(2a)] The rows $R_{i_2,k_1}$, $R_{i_1,k_2}$, $R_{i_2,k_2}$ are not overloaded;
this eliminates at most $\dfrac{12\kappa + 3\theta}{\epsilon} n+3$ choices
since in the file layer $k_1$ (which contains the row $R_{i_2,k_1}$) and 
in the row layer $i_1$ (which contains the row $R_{i_1,k_2}$) 
and in the row block which contains the row $R_{i_1,k_1}$ 
(which also contains the row $R_{i_2,k_2}$), 
there are in total at most $\dfrac{4\kappa + \theta}{\epsilon} n+1$ overloaded rows.
Similarly, we need that the files $F_{i_1,j_2}$, $F_{i_2,j_1}$,
$F_{i_2,j_2}$ are not overloaded; 
this eliminates at most $\dfrac{12\kappa + 3\theta}{\epsilon} n+3$ choices.

The columns $C_{j_2,k_1}$, $C_{j_1,k_2}$ are not overloaded; 
this eliminates at most $\dfrac{8\kappa + 2\theta}{\epsilon} n+2$ choices
since in the file layer $k_1$ (which contains the column 
$C_{j_2,k_1}$) and in the column layer $j_1$ 
(which contains the column $C_{j_1,k_2}$); there are in total 
at most $\dfrac{4\kappa + \theta}{\epsilon} n+1$ overloaded columns.

The two half columns of the column $C_{j_2,k_2}$ are not overloaded; 
this eliminates at most $\dfrac{6\kappa + 2\theta}{\epsilon} n$ choices
since in the first half column block determined by $(i_1,j_1,k_1)$ (which contains one half column of $C_{j_2,k_2}$)
and in the second half column block determined by $(i_1,j_1,k_1)$ (which contains the other half column of $C_{j_2,k_2}$);
there are in total at most $\dfrac{3\kappa + \theta}{\epsilon} n$ overloaded half columns.
Note that this condition also implies that the column $C_{j_2,k_2}$ is not overloaded,
since if $C_{j_2,k_2}$ is overloaded, then it contains at least $\epsilon n$ 
cells used  in $S$-swap,
which implies that one half column of $C_{j_2,k_2}$ must contain 
at least $\epsilon n /2$ cells used in $S$-swap, a contradiction.

\item[(2b)] The transversal-set $\mathfrak{t}_1$ containing $(i_1,j_2,k_1)$ and 
$(i_2,j_1,k_2)$ is not overloaded; this eliminates at most $\dfrac{4\kappa + \theta}{\epsilon} n+1$ 
choices, since in the row block which contains the row $R_{i_1,k_1}$ 
(which also contains the transversal-set $\mathfrak{t}_1$), there are at 
most $\dfrac{4\kappa + \theta}{\epsilon} n+1$ overloaded transversal-sets. 
Similarly, we need that the transversal-set containing $(i_2,j_2,k_1)$ 
and $(i_1,j_1,k_2)$ is not overloaded; this eliminates at most $\dfrac{4\kappa + \theta}{\epsilon} n+1$ choices.

The two half transversal-sets of the transversal-set $\mathfrak{t}_2$ 
containing $(i_2,j_1,k_1)$ 
and $(i_1,j_2,k_2)$, respectively, are not overloaded (this condition implies that
$\mathfrak{t}_2$ is not overloaded); this eliminates at most $\dfrac{6\kappa + 2\theta}{\epsilon} n$ 
choices, since in the first half transversal block determined by $(i_1,j_1,k_1)$
(which contains one half transversal-set of $\mathfrak{t}_2$) and 
in the second half transversal block determined by $(i_1,j_1,k_1)$
(which contains the other half transversal-set of $\mathfrak{t}_2$), there are at 
most $\dfrac{3\kappa + \theta}{\epsilon} n$ overloaded half transversal-sets.

\item[(2c)] The symbol-set $\mathfrak{s}_1$ containing $(i_2,j_1,k_2)$ and $(i_1,j_2,k_2)$ is not overloaded;
this eliminates at most $\dfrac{4\kappa + \theta}{\epsilon} n+1$ choices,
since in the symbol block which contains $(i_1,j_1,k_1)$ (which also contains the symbol-set $\mathfrak{s}_1$), there are 
at most $\dfrac{4\kappa + \theta}{\epsilon} n+1$ overloaded symbol-sets. Similarly, we need that the symbol-set
containing $(i_2,j_2,k_1)$ and $(i_1,j_2,k_2)$, and
the symbol-set containing $(i_2,j_2,k_1)$ 
and $(i_2,j_1,k_2)$ are not overloaded,
this eliminates at most $\dfrac{8\kappa + 2\theta}{\epsilon} n+2$ choices.
 
\item[(2d)] The symbol-set $\mathfrak{s}_2$ containing 
$(i_1,j_2,k_1)$ and $(i_2,j_1,k_1)$ is not overloaded. 
This eliminates at most $\dfrac{4\kappa + \theta}{\epsilon} n+1$ choices, since in the file layer $k_1$ 
(which contains the symbol-set $\mathfrak{s}_2$), there are at most $\dfrac{4\kappa + \theta}{\epsilon} n+1$ overloaded symbol-sets. 
Similarly, we need that the symbol-set containing $(i_1,j_1,k_2)$ and $(i_1,j_2,k_1)$, 
the symbol-set containing $(i_2,j_1,k_1)$ and $(i_1,j_1,k_2)$ are not overloaded. 
This eliminates at most $\dfrac{8\kappa + 2\theta}{\epsilon} n+2$ choices.

\item[(2e)] The symbol-set $\mathfrak{s}_3$ containing 
$(i_1,j_2,k_1)$ and $(i_2,j_2,k_2)$ is not overloaded. 
This eliminates at most $\dfrac{4\kappa + \theta}{\epsilon} n+1$ 
choices, since in the symbol-column block which contains $(i_1,j_1,k_1)$ 
(which also contains symbol-set $\mathfrak{s}_3$), there are at most $\dfrac{4\kappa + \theta}{\epsilon} n+1$ overloaded symbol-sets. 
Similarly, we need that the symbol-set containing $(i_1,j_1,k_2)$ and $(i_2,j_2,k_2)$, is not overloaded;
this eliminates at most $\dfrac{4\kappa + \theta}{\epsilon} n+1$ choices.

The two half symbol-sets of the symbol-set $\mathfrak{s}_4$ containing
$(i_2,j_1,k_1)$ and $(i_2,j_2,k_2)$ are not overloaded
(this condition implies that $\mathfrak{s}_4$ is not overloaded);
this eliminates at most $\dfrac{6\kappa + 2\theta}{\epsilon} n$ 
choices, since in the first half symbol-row block determined by $(i_1,j_1,k_1)$
(which contains one half symbol-set of $\mathfrak{s}_4$) and 
in the second half symbol-row block determined by $(i_1,j_1,k_1)$
(which contains the other half symbol-set of $\mathfrak{s}_4$), there are at 
most $\dfrac{3\kappa + \theta}{\epsilon} n$ overloaded half symbol-sets.

\end{itemize}
So in total, this eliminates at most $\dfrac{90\kappa + 24\theta}{\epsilon} n+18$ choices.
Note that with this strategy for including subcubes in $S$, after completing the construction of $S$, every row,
column, file, transversal-set, and symbol-set contains at most $2 \kappa n +(\epsilon n-1)+2=2 \kappa n +\epsilon n+1$ cells that are used in $S$-swap.

\item[(3)] Except for $(i_1,j_1,k_1)$, none of the cells in $\mathcal{C}$ are conflicts or have been used before in $S$-swap.
\begin{itemize}

\item[(3a)] The cell $(i_2,j_1,k_1)$ is not a conflict and has not been
used before in $S$-swap; 
this eliminates at most $3\kappa n +\epsilon n+1$ choices since the column 
$C_{j_1,k_1}$ contains at most $\kappa n$ conflict cells and at most $2 \kappa n +\epsilon n+1$ cells that are used in $S$-swap. 
Similarly, we need that the cell $(i_1,j_2,k_1)$ and the cell $(i_1,j_1,k_2)$ 
are not conflicts and has not used before in $S$-swap; in total, 
this eliminates at most $6\kappa n +2\epsilon n+2$ choices.

\item[(3b)] The cell $(i_1,j_2,k_2)$ is not a conflict and has not been
used before in $S$-swap. This eliminates at most 
$3\kappa n +\epsilon n+1$ choices, since in the symbol-set 
in row layer $i_1$ that contains the cell
$(i_1,j_1,k_1)$, there are at most $\kappa n$ conflict cells and at most $2 \kappa n +\epsilon n+1$ cells that have been used in $S$-swap. 
Similarly, we need that the cell 
$(i_2,j_1,k_2)$ and the cell $(i_2,j_2,k_1)$ are not conflicts and has not been
used before in $S$-swap; in total, this eliminates at most 
$6\kappa n +2\epsilon n+2$ choices.

\item[(3c)] The cell $(i_2,j_2,k_2)$ is not a conflict and has not been
used before in $S$-swap. This eliminates at most $3\kappa n +\epsilon n+1$ choices since in the transversal-set containing the cell $(i_1,j_1,k_1)$, there are at most $\kappa n$ conflict cells and at most $2 \kappa n +\epsilon n+1$ cells that are used in $S$-swap. 
\end{itemize}
So in total, this eliminates at most $21\kappa n +7\epsilon n+7$ choices. 
\end{itemize}

It follows that we have at least
$$\alpha n-21\kappa n -7\epsilon n-\dfrac{90\kappa}{\epsilon}n-\dfrac{24\theta}{\epsilon} n-\dfrac{544\kappa}{\theta} n-25$$
choices for an allowed subcube $\mathcal{C}$ which contains $(i_1,j_1,k_1)$. 
By \eqref{eq:param}, this expression is greater than zero if
$n$ is large enough,
so we can conclude that there is a subcube satisfying these conditions. Thus we may construct the set $S$ by iteratively adding disjoint allowed subcubes such that each subcube contains exactly one conflict cell.

After this process terminates, we have a set $S$ of disjoint subcubes; 
we swap on all subcubes in $S$ to obtain the Latin cube $L'$. 
Hence, we conclude that we can obtain a Latin cube $L'$ that avoids $A$.
\end{proof}


\begin{thebibliography}{99}
\addcontentsline{toc}{chapter}{Bibliography}


\bibitem{Andren2010latin}
L.J Andr\'{e}n, 
{\em On latin squares and avoidable arrays}, 
Ph.D. thesis, Ume\r{a}, Instituionen f\"or matematik och matematisk statisitik, 2010.

\bibitem{AndrenCasselgrenOhman}
L.J. Andr\'{e}n, C.J. Casselgren, L.-D. \"Ohman,
Avoiding arrays of odd order by Latin squares,
{\em Combinatorics, Probability and Computing} 22 (2013), 184--212.

\bibitem{AndrenCasselgrenMarkstrom}
L.J. Andr\'{e}n, C.J Casselgren, K. Markstr\"om,
Restricted completion of sparse Latin squares, to appear in
{\em Combinatorics, Probability and Computing}.

\bibitem{Lina}
L.J. Andr\'{e}n,
Avoiding $(m,m,m)$-arrays of order $n=2^k$,
{\em The Electronic Journal of combinatorics} 19 (2012), 11 pp.


\bibitem{Padraic}
P. Barlett
Completions of $\epsilon$-Dense Partial Latin Squares
{\em Journal of Combinatorial Designs} 
21 (2013), 447--463.


\bibitem{Bri}
 T. Britz, N.J. Cavenagh,
Maximal partial Latin cubes,
{\em Electron. J. Comb.,} 22 (2015), 17 pp.


\bibitem{Bryant}
 D. Bryant, N.J. Cavenagh, B. Maenhaut, K. Pula, I. Wanless. 
Non-extendible latin cuboids,
{\em SIAM J. Discrete Math.,} 26:239--249, 2012.

\bibitem{hypercube}
C.J Casselgren, K. Markstr\"om, L.A. Pham,
Restricted extension of sparse partial edge colorings of hypercubes,
submitted, preprint available on Arxiv.


\bibitem{JohanKlasLan}
C.J Casselgren, K. Markstr\"om, L.A. Pham,
Latin cubes with forbidden entries,
{\em The electronic journal of combinatorics}, 26 (2019), 18 pp.


\bibitem{Cruse}
A.B. Cruse,
On the finite completion of partial Latin cubes,
{\em J. Combin. Theory Ser. A}, 17, 112--119, 1974.

\bibitem{CulterOhman}
J. Cutler and L.-D. \"Ohman,
Latin squares with forbidden entries,
{\em Electron. J.Combin.}, 13, (2006) R47, 9 pp.

\bibitem{DenleyOhman}
T. Denley, L.-D. \"Ohman
Extending partial Latin cubes,
{\em Ars Combinatoria,} 113 (2014), 405–-414.

\bibitem{EGHKPS}
K.~{Edwards}, A.~{Gir{\~a}o}, J.~{van den Heuvel}, R.~J. {Kang}, G.~J. {Puleo},
and J.-S. {Sereni}, \emph{{Extension from Precoloured Sets of Edges}},
Electron. J. Combin. 25 (2018), 28 pp.

      
\bibitem{Haggkvist}
R.  H\"aggkvist.   A  note  on  Latin  squares  with  restricted  support,
{\em Discrete  Math.}, 75(1-3):253--254, 1989.  Graph theory and 
combinatorics (Cambridge, 1988)		
			
\bibitem{KuhlDenley}
J. Kuhl, T. Denley,
Some partial Latin cubes and their completions,
{\em European Journal of Combinatorics,} 32 (2011) 1345-1352.


\bibitem{MckWan}
B. McKay, and I. Wanless,
A Census of Small Latin Hypercubes,
{\em SIAM Journal on Discrete Mathematics,} 22 (2008) 719-736.



\end{thebibliography}
\end{document}